\documentclass[a4paper,11pt]{article}

\usepackage{titlesec}
\usepackage{amsthm}
\usepackage{amssymb}
\usepackage{color}
\usepackage{mathtools}
\usepackage{geometry}
\usepackage{multirow}
\usepackage{multicol}
\usepackage{rotating} 
\usepackage{amsmath}
\usepackage{booktabs} 

\usepackage{algorithm2e}

\usepackage{verbatim}
\numberwithin{equation}{section}

\newtheorem{theorem}{Theorem}[section]

\newtheorem{lemma}[theorem]{Lemma}

\theoremstyle{definition}

\newtheorem{remark}[theorem]{Remark}
\newcommand{\bfs}[1]{{\boldsymbol #1}}
\usepackage{graphicx}
\usepackage{subcaption}

\usepackage{array}
\newcolumntype{C}[1]{>{\centering\arraybackslash}m{#1}}
\usepackage{graphics}
\usepackage{amssymb, latexsym, amsmath, amsfonts, epsfig}
\usepackage{bm}
\usepackage{graphicx}
\usepackage{color}
\usepackage{epstopdf}
\usepackage{comment}
\usepackage{soul}
\usepackage[normalem]{ulem}
\usepackage{float}
\usepackage{MnSymbol}

\usepackage{caption}
\date{\vspace{-6ex}}

\begin{document}

\newcommand{\Question}[1]{{\marginpar{\color{blue}\footnotesize #1}}}
\newcommand{\rev}[1]{{\color{blue}#1}}
\newcommand{\ale}[1]{{\color{magenta} #1}}

\newif \ifNUM \NUMtrue

\title{Generalised Soft Finite Element Method for Elliptic Eigenvalue Problems}
\author{
Jipei Chen\thanks{School of Computing, Australian National University, Canberra, ACT 2601, Australia. E-mail address: Jipei.Chen@anu.edu.au}
\and
Victor M. Calo\thanks{School of Electrical Engineering, Computing and Mathematical Sciences, Curtin University, Perth, WA 6102, Australia. E-mail address: Victor.Calo@curtin.edu.au}\\
\and
Quanling Deng\thanks{Corresponding Author. School of Computing, Australian National University, Canberra, ACT 2601, Australia. E-mail address: Quanling.Deng@anu.edu.au} 
}

\maketitle

\begin{abstract}
The recently proposed soft finite element method (SoftFEM)  reduces the stiffness (condition numbers), consequently improving the overall approximation accuracy. The method subtracts a least-square term that penalizes the gradient jumps across mesh interfaces from the FEM stiffness bilinear form while maintaining the system's coercivity. Herein, we present two generalizations for SoftFEM that aim to improve the approximation accuracy and further reduce the discrete systems' stiffness. Firstly and most naturally, we generalize SoftFEM by adding a least-square term to the mass bilinear form. Superconvergent results of rates $h^6$ and $h^8$ for eigenvalues are established for linear uniform elements; $h^8$ is the highest order of convergence known in the literature. Secondly, we generalize SoftFEM by applying the blended Gaussian-type quadratures. We demonstrate further reductions in stiffness compared to traditional FEM and SoftFEM. The coercivity and analysis of the optimal error convergences follow the work of SoftFEM. Thus, this paper focuses on the numerical study of these generalizations. For linear and uniform elements, analytical eigenpairs, exact eigenvalue errors, and superconvergent error analysis are established. Various numerical examples demonstrate the potential of generalized SoftFEMs for spectral approximation, particularly in high-frequency regimes. 

\textbf{Mathematics Subjects Classification}: 65N30, 65N35, 35J05
\end{abstract}
%


\section{Introduction} 
\label{sec:intr}

Galerkin finite element methods have been extensively utilized for spectral approximation of the second-order elliptic eigenvalue problems. 
For foundational work in this field, we refer readers to pioneering studies by Vainikko \cite{vainikko1964asymptotic,vainikko1967speed}, Bramble and Osborn \cite{bramble1973rate}, Strang and Fix \cite{strang1973analysis}, Osborn \cite{osborn1975spectral}, Descloux et al. \cite{descloux1978spectral,descloux1978spectral2}, and Babu\v{s}ka and Osborn \cite{babuvska1991eigenvalue}, along with more recent reviews such as \cite{boffi2010finite,Ern_Guermond_FEs_II_2020}. Alternative approaches like mixed finite element methods \cite{canuto1978eigenvalue, mercier1981eigenvalue, mercier1978eigenvalue}, discontinuous Galerkin methods \cite{antonietti2006discontinuous,giani2015hp}, hybridizable discontinuous Galerkin methods \cite{cockburn2010hybridization, gopalakrishnan2015spectral}, hybrid high-order methods \cite{calo2019spectral,carstensen2021guaranteed}, and virtual element methods \cite{gardini2017virtual} have also been explored for these problems. While all these methods yield optimally convergent approximations, their accuracy primarily holds for the lower part of the spectrum, as the eigenfunctions tend to oscillate more at higher frequencies.
A method to address high-frequency errors is to apply high-order continuous basis functions such as the B-splines or non-uniform rational B-splines (NURBS).
This leads to the development of isogeometric analysis (IGA) \cite{cottrell2006isogeometric} that delivers a more accurate approximation in the upper part of the spectrum (see also \cite{deng2018dispersion,calo2019dispersion,deng2021boundary} for some recent improvements on the subject). 
Additionally, this research has led to advancements on the quadrature blending techniques, as detailed in \cite{CALO2017798, calo2019dispersion, PUZYREV2017421}. The core advantage of this approach is its proficiency in minimizing dispersion errors inherent in numerical integration. 

SoftFEM \cite{DENG2021119} improves the spectral approximation of the Galerkin FEM, with a focus on reducing the stiffness in the discrete spectral problem. Consequently, it also improves accuracy in the higher spectrum range. 
The idea was extended to the IGA setting, and SoftIGA was developed to improve the spectral approximation \cite{deng2023softiga}.

\begin{figure}[h!]
    \centering

    \includegraphics[width=0.45\textwidth]{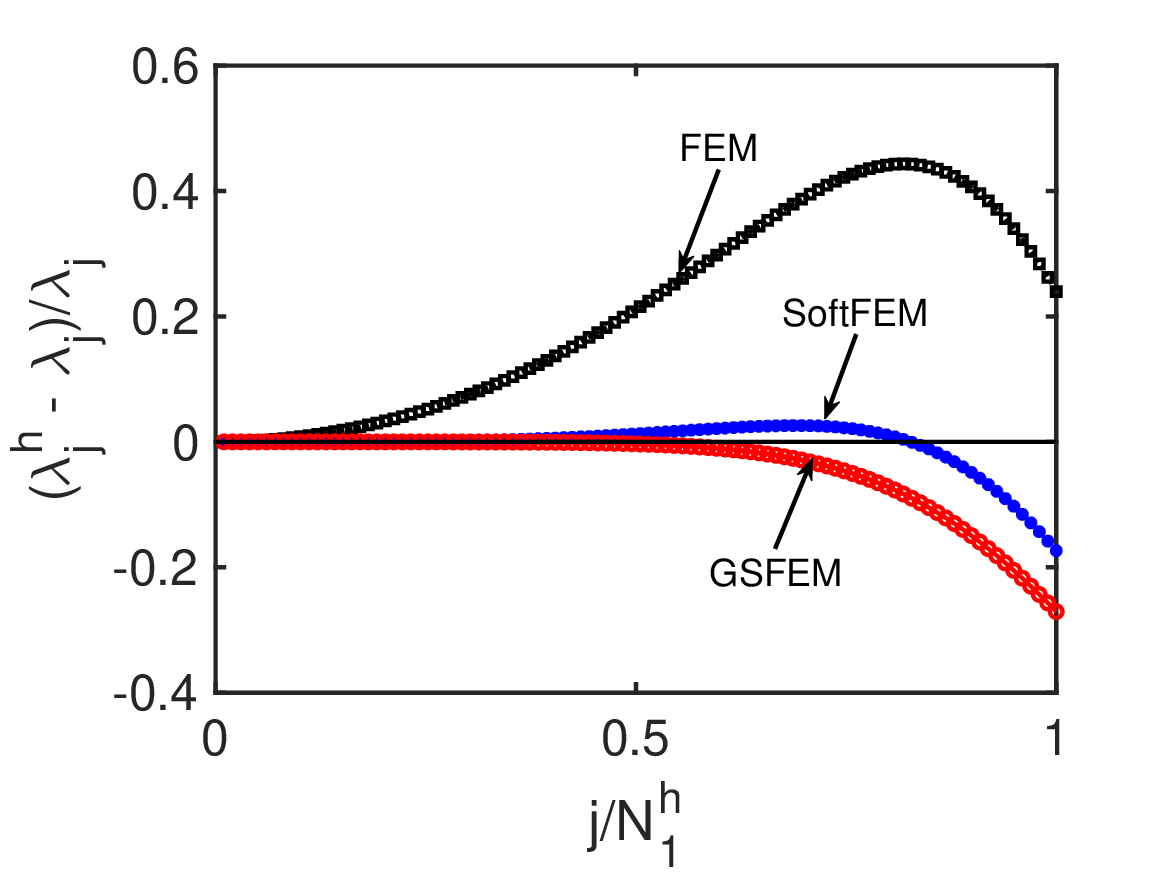}
    \includegraphics[width=0.45\textwidth]{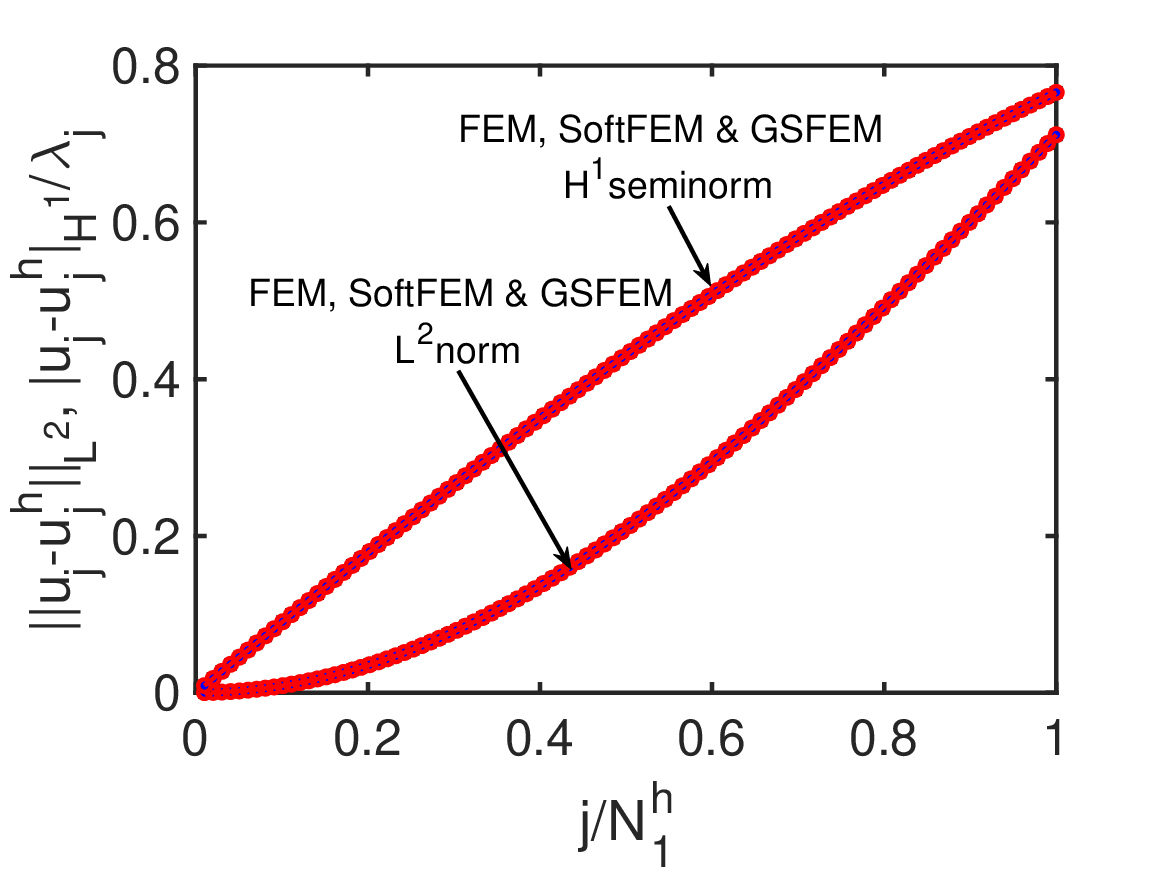}
    \includegraphics[width=0.45\textwidth]{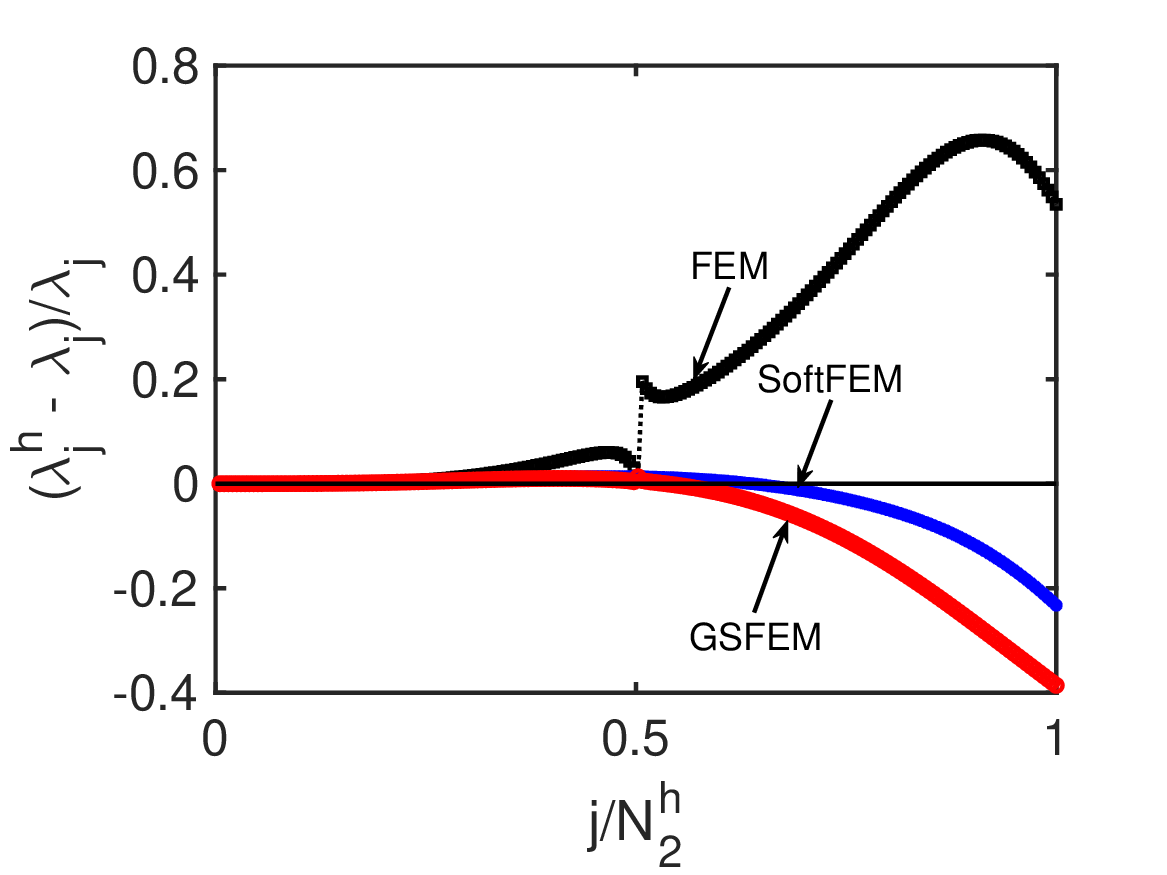}
    \includegraphics[width=0.45\textwidth]{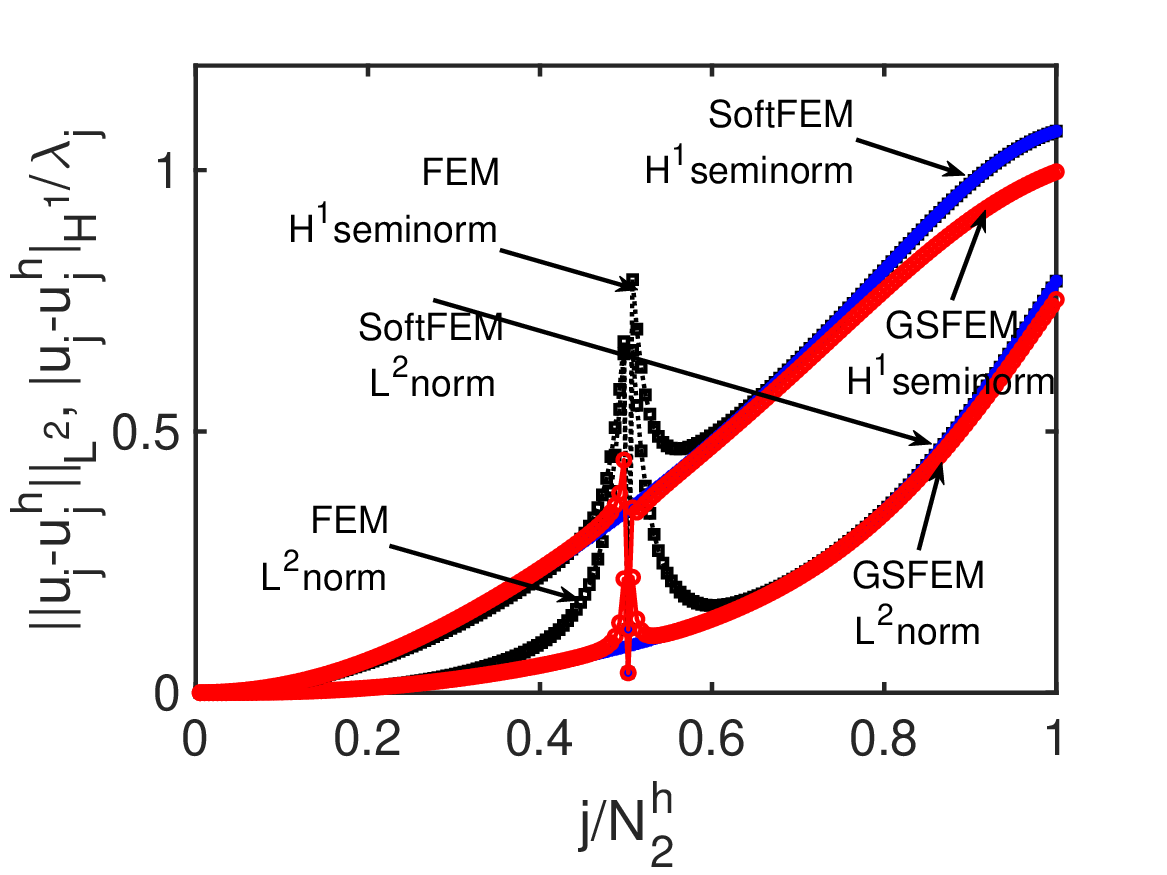}
    \includegraphics[width=0.45\textwidth]{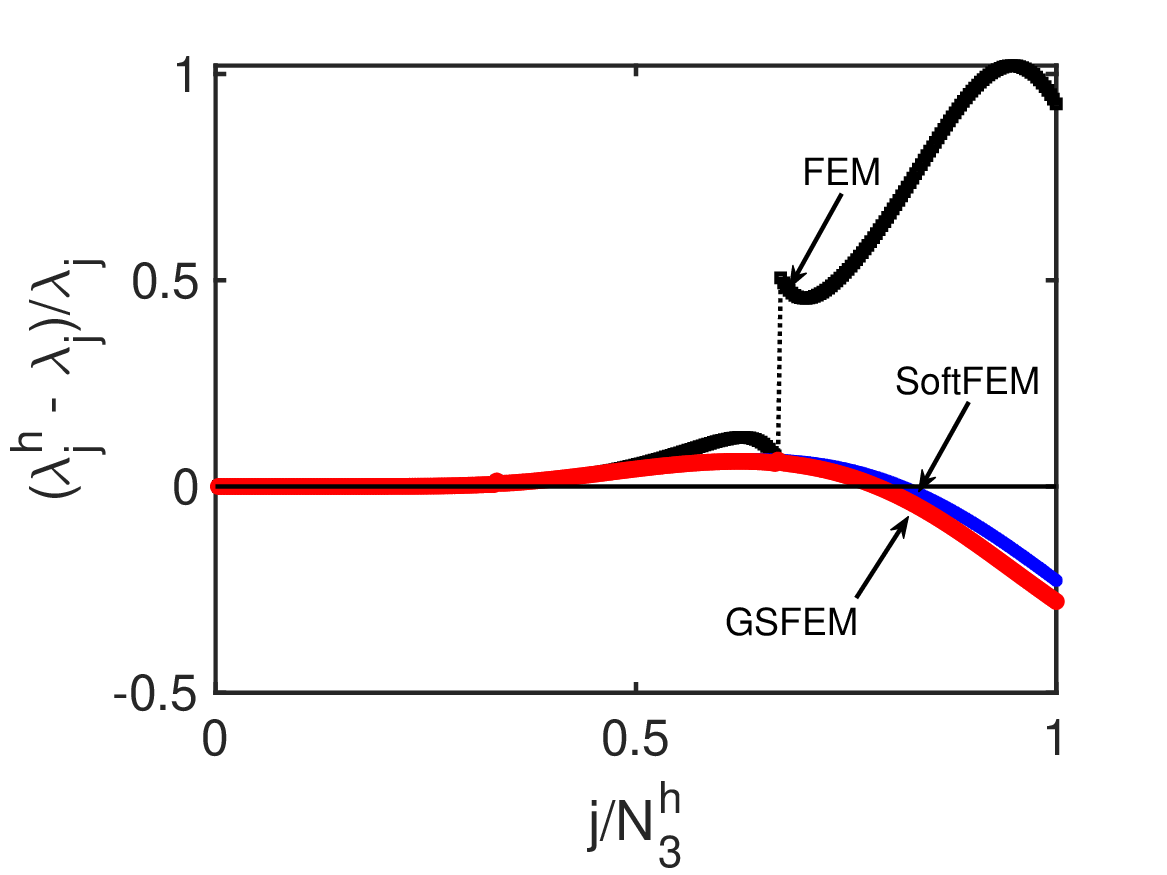}
    \includegraphics[width=0.45\textwidth]{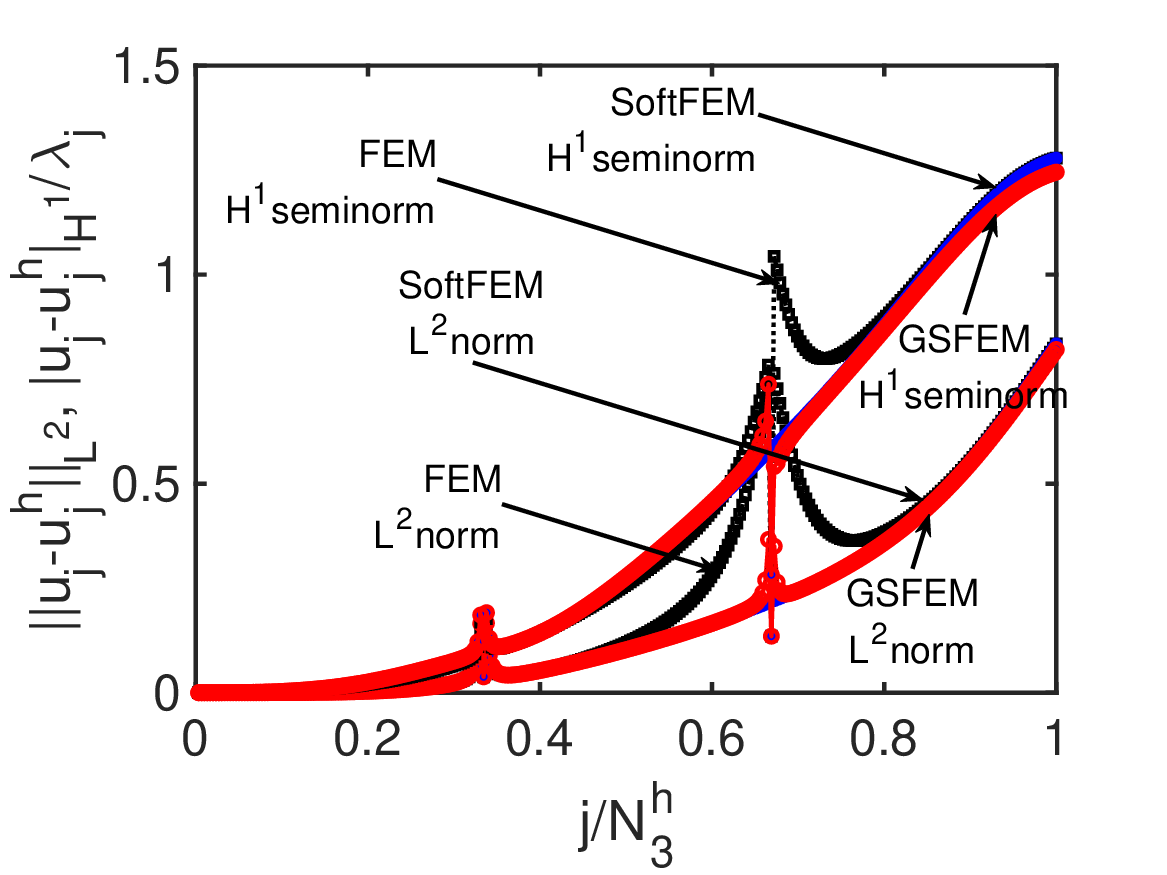}
    \caption{Comparison of Galerkin FEM (black), SoftFEM (blue), and GSFEM (red) for the spectral approximation of the Laplacian eigenvalue problem in 1D using $N^h=100$ elements with polynomial degrees $p\in \{1, 2, 3\}$ (from top to bottom row). }
    \label{fig:GSFEM_comparisionEV&E}
\end{figure}

In this study, we generalize SoftFEM in two ways: adding the auxiliary jump-penalty term to the mass bilinear form and applying blending quadrature rules. Our paramount goal is to reduce the condition numbers further and improve numerical accuracy by minimal additional computation.
SoftFEM reduces the discrete stiffness of the system by subtracting an auxiliary bilinear form derived from a least-squares penalty on the gradient jumps across element interfaces from the standard FEM stiffness. 
Hence, $a(\cdot, \cdot) \to a(\cdot, \cdot) - \eta_K s(\cdot, \cdot)$
where $a(\cdot, \cdot)$ is the standard FEM stiffness bilinear form, $s(\cdot, \cdot)$ is the auxiliary bilinear form and $\eta_K$ is the softness parameter. 
Herein, $s(\cdot, \cdot)$ denotes the bilinear form penalizing gradient discontinuities across element interfaces. 
The first idea to generalize SoftFEM is to apply the auxiliary bilinear form also to the standard FEM mass bilinear form.
In other words, the new mass bilinear form is then
\begin{equation*}
b_{g}(\cdot, \cdot) := b(\cdot, \cdot) + \eta_{M} s(\cdot, \cdot),
\end{equation*}
where $b(\cdot, \cdot)$ denotes the standard Galerkin FEM mass bilinear form and $\eta_{M}$ is a parameter.
Since the auxiliary bilinear form is positive semi-definite and the eigenvalues are approximated as Rayleigh quotients $\frac{a(\cdot, \cdot)}{b(\cdot, \cdot)}$. 
We add to the mass matrix bilinear form (i.e., use ``+" to reduce the system's overall stiffness further). 
We refer to this new method based on FEM as \textit{generalized SoftFEM} (GSFEM).

To provide readers with an initial glimpse of the advantages of GSFEM compared to Galerkin FEM and SoftFEM,
Figure~\ref{fig:GSFEM_comparisionEV&E} shows the comparison of eigenvalue and eigenfunction errors of these three methods 
for the spectral approximation of the 1D Laplace eigenvalue problem on $[0,1]$ with homogeneous Dirichlet boundary conditions.
Therein, we show the case of $p$-th order elements with $p\in\{1,2,3\}$ and $N^h=100$ elements. 
The advantages of GSFEM over softFEM and FEM is clear, particularly in the high-frequency regimes. 
The condition numbers (or stiffness as characterized by the ratio of maximal eigenvalue over minimal eigenvalue) are further reduced.


Another way to generalize SoftFEM is to apply blending quadratures (BQs) for the bilinear inner products. 
Quadrature blending techniques have been widely used to improve spectral approximation. 
For example, the work \cite{Ainsworth10} established optimally blended quadratures for optimally blending the spectral and finite element methods. 
The series of papers \cite{puzyrev2017dispersion, deng2018dispersion, calo2019dispersion} developed optimally-blended quadratures for IGA elements in various settings. 
These works demonstrated the advantageous use of blending quadratures. 
Alternatively, this quadrature-blending technique allows SoftFEM to improve the spectral approximation. 
These blending quadratures further reduce the stiffness of the GSFEM and SoftFEM discretized systems.

The rest of the paper is organized as follows: Section~\ref{sec:introOfSoftFEM} presents the problem statement, the standard Galerkin FEM, and the recently-developed SoftFEM. 
Section~\ref{sec:main} is devoted to the description of the main idea of our work. We introduce two lines of ideas for generalizing SoftFEM.
Section~\ref{sec:ana} presents analytical results of linear elements for the Laplace eigenvalue problem in 1D. This includes analysis of superconvergent eigenvalue errors. 
Finally, comprehensive numerical experiments focusing on the Dirichlet eigenvalue problems are presented in  Section~\ref{sec:numExperiments}. We also analyse problems in multiple dimensions and problems with inhomogeneous diffusion. In all the examples, optimal convergence rates are demonstrated. The numerical experiments demonstrate the advantages of employing GSFEM with blending quadratures, as compared to the standard FEM and SoftFEM. 
Concluding remarks are presented in Section~\ref{sec:conclusion}.

\section{Problem Statement and SoftFEM}
\label{sec:introOfSoftFEM}

In this section, we introduce the idea of SoftFEM, which serves as a foundation for the new methods we will explore in Section ~\ref{sec:main}. The subsequent subsections start with a problem statement, followed by a description of the standard Galerkin FEM, and then present the main idea of SoftFEM.

\subsection{Problem Statement}
\label{sec:problem}

Consider $\Omega$ as a bounded and open subset of $\mathbb{R}^d$ for $d\ge1$, characterized by a Lipschitz continuous boundary denoted as $\partial \Omega$. 
We adopt conventional notations used for Lebesgue and Sobolev spaces. For any measurable subset $S \subseteq \Omega$, the $L^2$ inner product and norm are represented as $(\cdot, \cdot)_S$ and $\|\cdot\|_S$, respectively. This notation extends to vector-valued fields similarly. Furthermore, for any integer $m \ge 1$, the norms and seminorms in the Sobolev space $H^m$ are denoted by $\|\cdot\|_{H^m(S)}$ and $|\cdot|_{H^m(S)}$, respectively.

We consider the following second-order elliptic eigenvalue problem with homogeneous Dirichlet boundary conditions: 
Find the eigenpair $(\lambda, u) \in \mathbb{R}^+ \times H^1_0(\Omega)$ with $\|u\|_\Omega=1$ so that
\begin{equation} \label{eq:pde}
\begin{alignedat}{2}
- \nabla \cdot (\kappa \nabla u) & = \lambda u &\quad &\text{in $\Omega$}, \\
u & = 0 &\quad &\text{on $\partial \Omega$},
\end{alignedat}
\end{equation}
where $\kappa \in L^\infty(\Omega)$ denotes the diffusion coefficient, uniformly bounded from below away from zero. Here, we define $\kappa_{\min} := \text{ess\,inf}_{x \in \Omega}\kappa(x) > 0$.
When $\kappa = 1$, problem~\eqref{eq:pde} simplifies to the classic Laplace (Dirichlet) eigenvalue problem. 
The variational formulation of problem~\eqref{eq:pde} is written as

\begin{equation} \label{eq:vf}
a(u, w) = \lambda b(u, w), \quad \forall w \in H^1_0(\Omega),
\end{equation}
with the bilinear forms defined as
\begin{equation}
a(v, w) := (\kappa \nabla v, \nabla w)_\Omega,
\qquad
b(v, w) := (v, w)_\Omega.
\end{equation}

The eigenvalue problem described in equation~\eqref{eq:pde} is characterized by a countable series of positive real eigenvalues, $\lambda_j \in \mathbb{R}^+$, as detailed in~\cite[Sec. 9.8]{Brezis:11}. These eigenvalues are ordered in a non-decreasing sequence:
\begin{equation*}
0 < \lambda_1 < \lambda_2 \leq \lambda_3 \leq \cdots.
\end{equation*}
We sort the approximate eigenvalues in ascending order, accounting for their algebraic multiplicity.
A set of $L^2$-orthonormal eigenfunctions is associated with these eigenvalues, denoted as $u_j$. These functions satisfy the condition $(u_j, u_k) = \delta_{jk}$, where $\delta_{jk}$ represents the Kronecker delta. Considering the variational formulation in~\eqref{eq:vf}, these normalized eigenfunctions also have orthogonality in the energy inner product.

\subsection{Galerkin FEM}
\label{sec:GalerkinFEM}
We now introduce Galerkin Finite Element Method (FEM). First, we denote by $(\mathcal{T}_h)_{h>0}$ a partition of the domain $\Omega$. 
Each mesh element in this sequence is represented by $\tau$ with its diameter denoted by $h_\tau$ and an outward unit normal vector $\bfs{n}_\tau$. The maximum diameter among all elements in the mesh is denoted by $h = \max_{\tau \in \mathcal{T}_h} h_\tau$. 
Thus, we consider the tensor-product meshes, and the corresponding Galerkin FEM space $V^h_p$ is defined as:
\begin{equation}
V^h_p := \{ v_h \in C^0(\overline \Omega): v_h |_{\partial \Omega} = 0, \text{ and } \forall \tau \in \mathcal{T}_h, v_h|_\tau \in \mathbb{Q}_p(\tau) \},
\end{equation}
where $\mathbb{Q}_p(\tau)$ is a polynomial space of polynomials of order at most $p$ in each dimension. 
Notably, $V^h_p$ is a subset of $H^1_0(\Omega)$, aligning with established theoretical frameworks in finite element analysis.

The approximation of the eigenvalue problem~\eqref{eq:pde} by the Galerkin FEM entails finding a pair $(\lambda^h, u^h) \in \mathbb{R}^+ \times V^h_p$ with the normalization $\|u^h\|_\Omega = 1$ so that
\begin{equation} \label{eq:vfh}
a(u^h, w^h) = \lambda^h b(u^h, w^h), \quad \forall w^h \in V^h_p.
\end{equation}
To achieve an algebraic realization of equation~\eqref{eq:vfh}, a set of basis functions $\{\phi_j^h\}_{j\in\{1,\cdots, N^h_p\}}$ is selected for the space $V^h_p$, where $N^h_p := \text{dim}(V^h_p)$ typically represents the dimensionality of the approximation space $V^h_p$.

This approach results in the formulation of a generalized matrix eigenvalue problem (GMEVP), represented as:
\begin{equation} \label{eq:mevp}
\mathbf{K} \mathbf{U} = \lambda^h \mathbf{M} \mathbf{U},
\end{equation}
where $\mathbf{K}_{kl}:= a(\phi_l^h, \phi_k^h)$ and $\mathbf{M}_{kl}:= b(\phi_l^h, \phi_k^h)$ for all $k, l \in \{1,\cdots,N^h_p\}$, denote the entries of the stiffness and mass matrices, respectively. In this framework, $\mathbf{U} \in \mathbb{R}^{N^h_p}$ represents the eigenvector that gives the components of $u^h$ as per the chosen basis.

\subsection{SoftFEM}
\label{sec:softFEM}

SoftFEM is characterized by modifying the standard bilinear form $a(\cdot,\cdot)$. This modification involves subtracting a least-squares penalty term, which targets the discontinuities in the normal derivatives across mesh interfaces. To elucidate this concept, we first introduce some notations.

For each element $\tau$ within the mesh $\mathcal{T}_h$, we let $h_\tau^0$ denote the diameter of the element $\tau$. The set of mesh interfaces is denoted as $\mathcal{F}_h^i$. Each interface $F$ in $\mathcal{F}_h^i$ is defined as the intersection of the boundaries of two distinct mesh elements, $\tau_1$ and $\tau_2$, i.e., $F = \partial\tau_1 \cap \partial\tau_2$. For such interfaces, we specify:
\begin{equation}
\label{eq:def_F_based}
h_F := \min(h_{\tau_1}^0, h_{\tau_2}^0), \qquad \kappa_F := \min(\kappa_{\tau_1}, \kappa_{\tau_2}),
\end{equation}
where $\kappa_\tau$ is the essential infimum of $\kappa$ over $\tau$. This implies that $\kappa_F$ is the lower bound of the diffusion coefficient $\kappa$ across the elements sharing the interface $F$.

Furthermore, for a function $v^h \in V_p^h$, the jump in its normal derivative across an interface $F$ is defined as:
\begin{equation}
\lsem \nabla v^h \cdot \bfs{n} \rsem_F := \nabla v^h|_{\tau_1} \cdot \bfs{n}_{\tau_1} + \nabla v^h|_{\tau_2} \cdot \bfs{n}_{\tau_2}.
\end{equation}
To enhance readability, we omit the subscript $F$ where the context is clear.

In the framework of SoftFEM, the objective is to seek an approximation $(\lambda^h_{s}, u^h_{s}) \in \mathbb{R}^+ \times V^h_p$ with the normalization $\|u^h_{s}\|_\Omega = 1$ such that
\begin{equation} 
\label{eq:softFEM} 
a_{s}(u^h_{s}, w^h) = \lambda^h_{s} b(u^h_{s}, w^h), \quad \forall w^h \in V^h_p, 
\end{equation}
where $a_{s}(\cdot, \cdot)$, defined for all $v^h, w^h \in V^h_p$, is a modified bilinear form:
\begin{equation} 
\label{eq:softFEMbf} 
a_{s}(\cdot, \cdot) := a(\cdot, \cdot) - \eta_{K} s(\cdot, \cdot),
\end{equation}
with the least-squares penalty term given by 
\begin{equation} \label{eq:bfs}
s(v^h, w^h) := \sum_{F \in \mathcal{F}_h^i} \kappa_F h_F (\lsem \nabla v^h \cdot \bfs{n} \rsem, \lsem \nabla w^h \cdot \bfs{n} \rsem )_F.
\end{equation}
Here, $\eta_{K} \ge 0$ is a parameter referred to as the \textit{softness parameter}. The concept of \textit{softness} in SoftFEM is rooted in the reduction of system stiffness imparted by the term $- \eta_{K} s(\cdot, \cdot)$. 
The range of $\eta_{K}$ is typically set within $[0, \eta_{K,\max})$, where $\eta_{K,\max}$ depends upon the polynomial degree $p$ and the characteristics of the mesh elements, ensuring the coercivity of the bilinear form $a_{s}(\cdot, \cdot)$. In the scenario where $\eta_{K} = 0$, SoftFEM simplifies to the standard Galerkin FEM.

In parallel with the Galerkin FEM approach, the algebraic formulation of the SoftFEM model, as presented in~\eqref{eq:softFEM}, results in a GMEVP:
\begin{equation} 
\label{eq:npmevp} 
\mathbf{K}_{s} \mathbf{U}_{s} = \lambda^h_{s} \mathbf{M} \mathbf{U}_{s},
\end{equation} 
wherein $\mathbf{K}_{s}$ is defined as $\mathbf{K} - \eta_{K} \mathbf{S}$. Here, $\mathbf{S}_{kl}$, representing the entries of the penalty matrix, is computed as $s(\phi_l^h, \phi_k^h)$. 
$\mathbf{K}$ and $\mathbf{M}$ denote the stiffness and mass matrices respectively, as established in~\eqref{eq:mevp}. The eigenvector $\mathbf{U}_{s}$ encapsulates the coefficients of $u^h_{s}$ relative to the selected basis $\{\phi_j^h\}_{j\in\{1,\cdots,N^h_p\}}$ for the finite element space $V_p^h$.
SoftFEM solutions are in the same solution space as those of Galerkin FEM.

\section{Generalized SoftFEM}
\label{sec:main}
In this section, we introduce two generalizations of SoftFEM to reduce the discrete systems' stiffness further and improve the spectral approximation. 
One is to add the least-squares penalty term to the mass bilinear form while the other is to apply the blending quadratures. 
We close this section by a natural extension of combining these two ideas. 
As the numerical experiments show, this combination further improves the stiffness reduction (see  Section~\ref{sec:numExperiments}).

\subsection{Generalization by Adding Gradient-Jump Penalty in Mass} 
\label{sec:geneSoftFEM}

SoftFEM subtracts a least-squares penalty term on the jumps of the normal derivatives across mesh interfaces to the stiffness bilinear form \(a(\cdot,\cdot)\). 
Following this idea,
a natural way to generalize SoftFEM is to add a similar least-squares penalty term multiplied by a different parameter, denoted as $\eta_{M}$, to the mass bilinear form. 
With this in mind, the generalized SoftFEM (GSFEM) is to find 
an approximated eigenpair $(\lambda^h_{gs}, u^h_{gs}) \in \mathbb{R}^+ \times V^h_p$ with the normalization $\|u^h_{gs}\|_\Omega = 1$ such that
\begin{equation} 
\label{eq:gsFEM} 
a_{s}(u^h_{gs}, w^h) = \lambda^h_{gs} b_{gs}(u^h_{gs}, w^h), \quad \forall w^h \in V^h_p, 
\end{equation}
where $a_{s}(\cdot, \cdot)$ is defined in~\eqref{eq:softFEMbf} and
\begin{equation} 
\label{eq:bilinearOfGeneSoftFEM} 
b_{gs} (v^h,w^h) := b(v^h,w^h)  + \eta_{M} s_g(v^h,w^h), 
\end{equation}
with 
\begin{equation} 
s_g(v^h, w^h) := \sum_{F \in \mathcal{F}_h^i} \kappa_F h_F^3 (\lsem \nabla v^h \cdot \bfs{n} \rsem, \lsem \nabla w^h \cdot \bfs{n} \rsem )_F.
\end{equation}
We note that there is a scale difference between $s_g(\cdot, \cdot)$ defined here and $s(\cdot, \cdot)$ defined in~\eqref{eq:bfs}.
This leads to the GMEVP of the form
\begin{equation} 
\label{eq:GeneSoftFEM} 
\mathbf{K}_{s} \mathbf{U}_{gs} = \lambda^h_{gs} \mathbf{M}_{gs} \mathbf{U}_{gs}, 
\end{equation} 
where the eigenpair to be sought is denoted as $(\lambda^h_{gs}, \mathbf{U}_{gs})$.
The modified matrix \(\mathbf{M}_{gs} = \mathbf{M} + \eta_{M} \mathbf{S}_g\) is computed from the new mass bilinear form $b_{gs}(\cdot, \cdot)$. 
When $\eta_{M} = 0$, GSFEM reduces to SoftFEM.

\begin{remark}[GSFEM]
Mass lumping is sometimes used in engineering problems. GSFEM and SoftFEM result in the same mass matrix after lumping.
\end{remark}

\subsection{Generalization by Blending Quadratures}
\label{sec:softFEMBQ}

We use the blending quadratures to generalize SoftFEM. 
We refer to, for example, \cite{davis2007methods} for the details of Gauss-type quadratures. 
Herein, we consider blending Gauss-Legendre and Gauss-Lobatto quadratures with a weight $\alpha$.
We denote this method of SoftFEM with blending quadrature as SoftFEMBQ.

The mass matrix of SoftFEMBQ, denoted as \(\mathbf{M}_{sq}\), is formulated as a linear combination of Gauss-Legendre and Gauss-Lobatto quadrature-integrated mass matrices. 
For $p$-th order elements, let \(\mathbf{M}_{\text{G}}\) denote the mass matrix with entries integrated by using the $(p+1)$-point Gauss-Legendre quadrature rule
while
let \(\mathbf{M}_{\text{L}}\) denote the mass matrix with entries integrated by using the $(p+1)$-point Gauss-Lobatto quadrature rule.
The modified mass matrix \(\mathbf{M}_{sq}\) is then defined as:
\begin{equation} \label{eq:alphaM}
\mathbf{M}_{sq} = \alpha \mathbf{M}_{\text{G}} + (1-\alpha) \mathbf{M}_{\text{L}},
\end{equation}
where \(\alpha\) is a weight parameter. 
The choice of \(\alpha\) can be optimized to reduce the eigenvalue and eigenfunction errors and the discretized system's condition numbers. 
When $\alpha=1$, SoftFEMBQ reduces to SoftFEM.
In analogy to previous cases, the corresponding GMEVP can be written as
\begin{equation} 
\label{eq:softFEMBQ} 
\mathbf{K}_{s} \mathbf{U}_{sq} = \lambda^h_{sq} \mathbf{M}_{sq} \mathbf{U}_{sq}, 
\end{equation}
The mass matrix herein generally differs from its counterpart as in~\ref{eq:GeneSoftFEM}.

\begin{remark}[SoftFEMBQ]
Both GSFEM in subsection~\ref{sec:geneSoftFEM} and SoftFEMBQ aim to reduce the stiffness by modifying the mass bilinear form. 
They differ in the formulation. 
Both methods produce optimally convergence eigenpairs. 
They are combined as a new method to improve the spectral approximation further. 
\end{remark}

We generalize SoftFEM by combining these ideas and defining a new mass matrix as 
\begin{equation}
\mathbf{M}_{gsq} := \mathbf{M}_{sq} + \eta_{M} \mathbf{S}_g,
\end{equation}
where the soft parameter \(\eta_{M}\) and $\alpha$ play the role of combination scaling. We refer to this method as GSFEMBQ.
The corresponding GMEVP is expressed as
\begin{equation} 
\label{eq:GSFEMBQ} 
\mathbf{K}_{s} \mathbf{U}_{gsq} = \lambda_{gsq}^h \mathbf{M}_{gsq} \mathbf{U}_{gsq}.
\end{equation}

\begin{remark}[Alternative Combination]
These two ways of generalization can be combined differently by introducing a weighting parameter as in~\eqref{eq:alphaM}. 
That is, $\mathbf{M}_{gsq} := \gamma \mathbf{M}_{sq} + (1-\gamma) \mathbf{M}_{gs}$, 
where $0 \le \gamma\le 1$.
We observe that numerically, this way of combination leads to similar improvement of the spectrum as ~\eqref{eq:GSFEMBQ}.
However, this combination introduces one more parameter, which requires optimization. 
With the numerical insights in mind, we thus adopt the combination~\eqref{eq:GSFEMBQ}.
\end{remark}

\section{Analytical Results of Linear Elements 1D} \label{sec:ana}

The goal of this paper is to generalize SoftFEM further to reduce the stiffness of the FEM discretized system. 
The condition number characterizes the system stiffness:
\begin{equation}
\sigma_*:= \frac{\lambda^h_{*,\max}}{\lambda^h_{*,\min}},
\end{equation}
where $*$ is replaced for different methods (\textit{gs} for GSFEM, \textit{sq} for SoftFEMBQ, and \textit{gsq} for GSFEMBQ).  $\lambda^h_{*,\min} $ and $  \lambda^h_{*,\max}$ are the smallest and largest eigenvalues of the corresponding method, respectively. 
When $*$ is omitted, it refers to the case of the standard Galerkin FEM.
Following \cite{DENG2021119},
 the \textit{stiffness reduction ratio}, \textit{asymptotic stiffness reduction ratio},  \textit{stiffness reduction percentage}, and \textit{asymptotic stiffness reduction percentage} of the method $*$ relative to Galerkin FEM are defined, respectively, as 
\begin{equation} \label{eq:sr}
\begin{aligned}
\rho_* & := \frac{\sigma}{\sigma_*} = \frac{\lambda^h_{\max} }{ \lambda^h_{*,\max} } \cdot \frac{ \lambda^h_{*,\min} }{\lambda^h_{\min}}, 
&& \rho_{*,\infty}  := \lim_{h\to 0} \frac{\lambda^h_{\max}}{ \lambda^h_{*,\max}}, \\
\varrho_* & = 100 \frac{\sigma- \sigma_*}{\sigma}\,\% = 100(1-\rho^{-1}_*),
&& \varrho_{*,\infty}  := 100 (1-\rho_{*,\infty}^{-1}).
\end{aligned}
\end{equation}

In this section, we analyze the linear element approximation of problem~\eqref{eq:pde} with $\kappa = 1$ and $\Omega = (0, 1)$ in 1D. 
In such a case, the problem has the exact eigenpairs
$(\lambda_j = j^2 \pi^2, u_j(x) = \sqrt{2} \sin( j\pi x)), j = 1,2, \cdots.$
For analysis, we discretize the interval $\Omega = (0,1)$ into $N^h$ uniform elements with the mesh size $h = 1/N^h.$ 
We derive analytical eigenpairs in terms of the auxiliary parameters and the choices for superconvergences. 

Firstly, on a mesh with $N^h$ uniform elements in 1D, the standard FEM stiffness and mass matrices, as well as the gradient-jump matrix are well-known (see, e.g., \cite{DENG2021119})
\begin{equation} \label{eq:kms1d}
\mathbf{K} = 
\frac{1}{h}
\begin{bmatrix}
2 & -1 &  \\
-1 & 2 & -1 &  \\
   & \ddots & \ddots & \ddots &  \\
  &  & -1 & 2 & -1  \\
&& & -1 & 2  \\
\end{bmatrix}, 
\mathbf{M} = h
\begin{bmatrix}
 \frac{2}{3} & \frac{1}{6} \\[0.2cm]
\frac{1}{6} & \frac{2}{3} & \frac{1}{6} \\[0.2cm]
 & \ddots & \ddots & \ddots &  \\[0.2cm]
&  & \frac{1}{6} & \frac{2}{3} & \frac{1}{6}  \\[0.2cm]
&& & \frac{1}{6} & \frac{2}{3}  \\
\end{bmatrix},
\mathbf{S} = 
\frac{1}{h}
\begin{bmatrix}
5 & -4 & 1 \\
-4 & 6 & -4 & 1  \\
1 & -4 & 6 & -4 & 1 \\
   & \ddots & \ddots & \ddots &  \ddots & \ddots \\
  &  & 1 & -4 & 6 & -4  \\
& & & 1 & -4 & 5 \\
\end{bmatrix}, 
\end{equation}
which are of dimension ${(N^h-1) \times (N^h-1)}$. 
Herein, the basis functions associated with the boundaries are removed so that $V^H_1 \subset H^1_0(\Omega)$.

\subsection{GSFEM} \label{sec:an1}

\begin{lemma}[Analytical eigenvalues and eigenvectors] \label{lem:gsfem1d}
For GSFEM approximation with $N^h$ uniform elements, the GMEVP $\mathbf{K}_{s} \mathbf{U}_{gs} = \lambda^h_{gs} \mathbf{M}_{gs} \mathbf{U}_{gs}$ in~\eqref{eq:GeneSoftFEM} has eigenpairs
$(\lambda_{gs,j}^h, \mathbf{U}_{gs,j})$ for all $j\in\{1,\cdots,N^h-1\}$ with 
\begin{equation} \label{eq:evm1}
\begin{aligned}
\lambda_{gs,j}^h & = \frac{12}{h^2}
\frac{\big(1 - 2\eta_K + 2\eta_K\cos(t_j) \big) \sin^2(t_j/2)}
{2 + 18 \eta_M + (1-24\eta_M) \cos(t_j) + 6 \eta_M \cos(2t_j)}, \\
\mathbf{U}_{gs,j,k} & = c_j \sin(kt_j), \qquad k =1,\cdots, N^h-1,
\end{aligned}
\end{equation}
where $t_j:=j \pi h$, $\mathbf{U}_{gs,j,k}$ is the $k$-th component of the $j$-th eigenvector $\mathbf{U}_{gs,j}$, $c_j> 0$ is some normalization constant.
\end{lemma}
\begin{proof}
The matrices $\mathbf{K}_{s}$ and $\mathbf{M}_{gs}$ have the Toeplitz-plus-Hankel structure. 
The analytical eigenpairs of the GMEVP are obtained as an application of~\cite[Thm.~2.1]{deng2021analytical}.  
\end{proof}

\begin{remark}
    In the case of $\eta_M=0$, GSFEM reduces to SoftFEM, consequently, the result reduces to the analytical eigenpairs Eqn. 3.5 in Lemma 1 in \cite{DENG2021119}. 
    If $\eta_K=\eta_M=0$, then the result reduces to the well-known result for the standard Galerkin FEM; see, for example,~\cite[Sec. 2]{boffi2010finite} and~\cite[Sec. 4]{hughes2008duality}.
\end{remark}

\begin{theorem}[Eigenvalue superconvergence] \label{thm:gsfemp1}
Let $\lambda_j$ be the $j$-th exact eigenvalue of~\eqref{eq:pde} with $\kappa=1$ in $\Omega=[0,1]$ and let $\lambda_{gs,j}^h$ be the $j$-th approximate eigenvalue using linear GSFEM. Assume that $\eta_K = \frac{1}{12}$ and $\eta_M = \frac{1}{360}$. The following holds:
\begin{equation} \label{eq:scm1}
\frac{ | \lambda_{gs,j}^h - \lambda_j|}{\lambda_j}  < \frac{1}{3024} (j \pi h)^6, \qquad \forall j\in\{1,\cdots, N^h-1\}.
\end{equation}
\end{theorem}

\begin{proof}
The exact eigenvalues $\lambda_j=j^2\pi^2$.
The approximate eigenvalues $\lambda_{gs,j}^h$ are given in~\eqref{eq:evm1}. 
To motivate the choice of $\eta_K = \frac{1}{12}$ and $\eta_M = \frac{1}{360}$, we observe that 
applying a Taylor expansion to $\lambda_{gs,j}^h$ gives
\begin{equation*}
\frac{ \lambda_{gs,j}^h - \lambda_j}{ \lambda_j} = \frac{1-12\eta_K}{12} t_j^2 + \frac{1-360\eta_M}{360}t_j^4 + \Big( \eta_K\eta_M + \frac{\eta_K}{720} - \frac{\eta_M}{12} -\frac{17}{60480} \Big) t_j^6+ \mathcal{O}(t_j^8),
\end{equation*} 
where $t_j=j \pi h$. Setting the leading order terms to be zeros, i.e.,
\begin{align}
    \frac{1-12\eta_K}{12} & = 0, \\
    \frac{1-360\eta_M}{360} & = 0,
\end{align}
gives the optimal choice $\eta_K = \frac{1}{12}$ and $\eta_M = \frac{1}{360}$ for superconvergence. 
Consequently, the Taylor expansion reduces to 
\begin{equation*}
\frac{ \lambda_{gs,j}^h - \lambda_j}{ \lambda_j} = -\frac{1}{6048} t_j^6- \frac{1}{43200} t_j^8+ \mathcal{O}(t_j^{10}).
\end{equation*} 
For small $j$ so that $t_j<<1$, the inequality~\eqref{eq:scm1} holds as the error is dominated by $t_j^6$ ($t_j = j\pi h >1$ for large $j$).
In general, for any $j=1,2,\cdots, N^h-1$, we analyze
\begin{equation*}
\frac{ | \lambda_{gs,j}^h - \lambda_j |}{ \lambda_j} = \left| \frac{120 (5 + \cos(t_j)) \sin^2(t_j/2) }{t_j^2 (123 + 56\cos(t_j) + \cos(2t_j) )} -1 \right|.
\end{equation*} 
Since $t_j=j\pi h$ samples the interval $(0,\pi)$, we can consider a continuous variable $t\in (0,\pi)$ and prove more generally that 
\[
\left| \frac{120 (5 + \cos(t)) \sin^2(t/2) }{t^2 (123 + 56\cos(t) + \cos(2t) )} -1 \right| < \frac{1}{3024}t^6,
\]
or, equivalently, that 
$$ 
- t^8 \gamma < 3024 \big( 120 (5 + \cos(t_j)) \sin^2(t_j/2) -t^2 \gamma  \big) 
< 
t^8 \gamma, 
$$
for all $t\in (0,\pi)$ where $\gamma = 123 + 56\cos(t) + \cos(2t)$. 
The left-side inequality holds since  
$$
f(t) := t^8 \gamma + 3024 \big( 120 (5 + \cos(t_j)) \sin^2(t_j/2) -t^2 \gamma  \big) 
$$
is increasing on $(0,\pi)$ and $f(0)=0$.
Similarly, the right-side inequality holds since 
$$
g(t) := t^8 \gamma - 3024 \big( 120 (5 + \cos(t_j)) \sin^2(t_j/2) -t^2 \gamma  \big) 
$$
is increasing on $(0,\pi)$ and $f(0)=0$.
This completes the proof.
\end{proof}

\begin{remark}
    We note that the coefficient $\frac{1}{3024}$ in~\eqref{eq:scm1} is not sharp. However, this constant is enough to demonstrate the superconvergence of the GSFEM for all the approximated eigenvalues. 
\end{remark}

With the choice of $\eta_K = \frac{1}{12}$ and $\eta_M = \frac{1}{360}$,
the stiffness reduction ratio and the asymptotic stiffness reduction ratio of GSFEM are 
\begin{equation} 
\begin{aligned}
\rho_{gs}(h) & = \frac{5 + \cos(\pi h) }{5 - \cos(\pi h) }\cdot \frac{2 + \cos(\pi h) }{2 - \cos(\pi h) }
\cdot \frac{123 - 56\cos(\pi h) + \cos(2\pi h)}{123 + 56\cos(\pi h) + \cos(2\pi h) }, \\
\rho_{gs,\infty} & = \lim_{h \to 0} \rho_{gs}(h) = \frac{17}{10}.
\end{aligned}
\end{equation}
The asymptotic stiffness reduction percentage is $\varrho_{gs,\infty} = \frac{700}{17}$, meaning linear GSFEM reduces the stiffness of Galerkin FEM by about $41.2\%$. 

\begin{remark}
    Since $\rho_{gs,\infty} = \frac{17}{10} > \frac{3}{2} = \rho_{s,\infty}$, the asymptotic stiffness reduction ratio of GSFEM is greater than that of SoftFEM (See Eq. (3.6) in \cite{DENG2021119}). 
    Similarly, the asymptotic stiffness reduction percentage is $41.2\%$ larger than that of SoftFEM (33.3\%).
    Thus, GSFEM further reduces the stiffness of the SoftFEM. 
\end{remark}

\subsection{SoftFEMBQ} \label{sec:an2}
The above analysis extends to SoftFEMBQ. 
Firstly, the mass matrix $\bfs{M}_{sq}$ is a combination of the mass matrix in~\eqref{eq:kms1d} (by Gauss-Legendre rule) with the mass matrix by the  Gauss-Lobatto rule.
The mass matrix resulting from the Gauss-Lobatto rule is diagonal. 
This diagonal matrix is an identity matrix scaled by $h$ for linear uniform elements.
With all the matrix entries, we obtain the following analytical results. 

\begin{lemma}[Analytical eigenvalues and eigenvectors] \label{lem:sfembq1d}
For SoftFEMBQ approximation, the GMEVP $\mathbf{K}_{s} \mathbf{U}_{sq} = \lambda^h_{sq} \mathbf{M}_{sq} \mathbf{U}_{sq}$ in~\eqref{eq:softFEMBQ} has eigenpairs
$(\lambda_{sq,j}^h, \mathbf{U}_{sq,j})$ for all $j\in\{1,\cdots,N^h-1\}$ with 
\begin{equation} 
\begin{aligned}
\lambda_{sq,j}^h & = \frac{12}{h^2}
\frac{\big(1 - 2\eta_K + 2\eta_K\cos(t_j) \big) \sin^2(t_j/2)}
{3 - \alpha + \alpha \cos(t_j) }, \\
\mathbf{U}_{sq,j,k} & = c_j \sin(kt_j), \qquad k =1,\cdots, N^h-1,
\end{aligned}
\end{equation}
where $t_j:=j \pi h$, $\mathbf{U}_{sq,j,k}$ is the $k$-th component of the $j$-th eigenvector $\mathbf{U}_{sq,j}$, $c_j> 0$ is some normalization constant.
\end{lemma}

\begin{theorem}[Eigenvalue superconvergence] \label{thm:sfemobsp1}
Let $\lambda_j$ be the $j$-th exact eigenvalue of~\eqref{eq:pde} and let $\lambda_{sq,j}^h$ be the $j$-th approximate eigenvalue using linear SoftFEMBQ. Assume that $\eta_K = \frac{1}{20}$ and $\alpha = \frac{4}{5}$. The following holds:
\begin{equation}
\frac{ | \lambda_{sq,j}^h - \lambda_j|}{\lambda_j}  < \frac{1}{1440} (j \pi h)^6, \qquad \forall j\in\{1,\cdots, N^h-1\}.
\end{equation}
\end{theorem}

\begin{proof}
The proof follows that of Theorem~\ref{thm:gsfemp1}. 
The Taylor expansion gives
\begin{equation*}
\begin{aligned}
    \frac{ \lambda_{sq,j}^h - \lambda_j}{ \lambda_j} 
    & = \Big( -\frac{1}{12} -\eta_K + \frac{\alpha}{6} \Big) t_j^2 + \frac{1 + 60\eta_K - 10 \alpha - 60 \alpha \eta_K + 10\alpha^2}{360}t_j^4 \\
    & \quad + 
    \frac{-3 -756\eta_K + 126 \alpha + 2520 \alpha \eta_K -420 \alpha^2 - 1680 \eta_K \alpha^2 + 280\alpha^3}{60480}  t_j^6+ \mathcal{O}(t_j^8),
\end{aligned}
\end{equation*} 
leading to the system of equations
\begin{align}
    -\frac{1}{12} -\eta_K + \frac{\alpha}{6} & = 0, \\
    \frac{1 + 60\eta_K - 10 \alpha - 60 \alpha \eta_K + 10\alpha^2}{360} & = 0,
\end{align}
whose solution $\eta_K = \frac{1}{20}$ and $\alpha = \frac{4}{5}$ provides the optimal choice for superconvergence. 
The rest of the proof follows similarly as a result of the inequality:
\[
\left| \frac{6 (9 + \cos(t)) \sin^2(t/2) }{t^2 (11 + 4\cos(t) )} -1 \right| < \frac{1}{1440}t^6,
\]
for $t\in (0, \pi)$.
\end{proof}

With the choice of $\eta_K = \frac{1}{20}$ and $\alpha = \frac{4}{5}$,
the stiffness reduction ratio and the asymptotic stiffness reduction ratio of SoftFEMBQ are 
\begin{equation} 
\begin{aligned}
\rho_{sq}(h) & = \frac{9 + \cos(\pi h) }{9 - \cos(\pi h) }\cdot \frac{2 + \cos(\pi h) }{2 - \cos(\pi h) }
\cdot \frac{11 - 4\cos(\pi h) }{11 + 4\cos(\pi h) }, \\
\rho_{sq,\infty} & = \lim_{h \to 0} \rho_{sq}(h) = \frac{7}{4}.
\end{aligned}
\end{equation}
The asymptotic stiffness reduction percentage is $\varrho_{sq,\infty} = \frac{300}{7}$, meaning linear SoftFEMBQ reduces the stiffness of Galerkin FEM by about $42.9\%$. 

\begin{remark}
    Since $\rho_{sq,\infty} = \frac{7}{4} > \rho_{gs,\infty} = \frac{17}{10} > \frac{3}{2} = \rho_{s,\infty}$, the asymptotic stiffness reduction ratio of SoftFEMBQ is greater than that of GSFEM which is greater than that of SoftFEM. 
    Similarly, the asymptotic stiffness reduction percentage of SoftFEMBQ is the largest among the three methods.  
\end{remark}

\subsection{Combination of GSFEM and SoftFEMBQ}
Lastly, in this section, we analyze an optimal combination of GSFEM and SoftFEMBQ.
Following the same analysis above, we derive analytical results below.
In particular, an extra superconvergent result of rate $\mathcal{O}(h^8)$ is achieved. 
To the best of our knowledge, this is the superconvergent result of the highest order in literature for linear elements.

\begin{lemma}[Analytical eigenvalues and eigenvectors] \label{lem:comb1d}
For GSFEMBQ approximation, the GMEVP $\mathbf{K}_{s} \mathbf{U}_{gsq} = \lambda^h_{gsq} \mathbf{M}_{gsq} \mathbf{U}_{gsq}$ in~\eqref{eq:GSFEMBQ} has eigenpairs
$(\lambda_{gsq,j}^h, \mathbf{U}_{gsq,j})$ for all $j\in\{1,\cdots,N^h-1\}$ with 
\begin{equation} 
\begin{aligned}
\lambda_{gsq,j}^h & = \frac{12}{h^2}
\frac{\big(1 - 2\eta_K + 2\eta_K\cos(t_j) \big) \sin^2(t_j/2)}
{3 + 18 \eta_M - \alpha + (\alpha-24\eta_M) \cos(t_j) + 6 \eta_M \cos(2t_j) }, \\
\mathbf{U}_{gsq,j,k} & = c_j \sin(kt_j), \qquad k =1,\cdots, N^h-1,
\end{aligned}
\end{equation}
where $t_j:=j \pi h$, $\mathbf{U}_{gsq,j,k}$ is the $k$-th component of the $j$-th eigenvector $\mathbf{U}_{gsq,j}$, $c_j> 0$ is some normalization constant.
\end{lemma}

\begin{theorem}[Eigenvalue superconvergence] \label{thm:combp1}
Let $\lambda_j$ be the $j$-th exact eigenvalue of~\eqref{eq:pde} and let $\lambda_{gsq,j}^h$ be the $j$-th approximate eigenvalue using linear GSFEMBQ. Assume that $\eta_K = \frac{31}{252}, \eta_M = \frac{23}{3780}$ and $\alpha = \frac{26}{21}$. The following holds:
\begin{equation}
\frac{ | \lambda_{gsq,j}^h - \lambda_j|}{\lambda_j}  < \frac{1}{30240} (j \pi h)^8, \qquad \forall j\in\{1,\cdots, N^h-1\}.
\end{equation}
\end{theorem}

\begin{proof}
The proof proceeds similarly to the proof of Theorem~\ref{thm:gsfemp1} and Theorem~\ref{thm:sfemobsp1}. 
We omit it here for simplicity. 
\end{proof}

\begin{remark}[Optimal choice of parameters]
The optimal choice $\eta_K = \frac{31}{252}, \eta_M = \frac{23}{3780}$ and $\alpha = \frac{26}{21}$
is the solution to the equations that the coefficients of the terms $t^2_j, t^4_j, t^6_j$ are zeros for the following Taylor expansion. 
\begin{equation*}
\begin{aligned}
    \frac{ \lambda_{gsq,j}^h - \lambda_j}{ \lambda_j} 
    & = \Big( -\frac{1}{12} -\eta_K + \frac{\alpha}{6} \Big) t_j^2 + \frac{1 + 60\eta_K - 360\eta_M - 10 \alpha - 60 \alpha \eta_K + 10\alpha^2}{360}t_j^4 \\
    & \quad + 
    \frac{-3 -756\eta_K +15120\eta_M + 126 \alpha + 60480\eta_K \eta_M +  2520 \alpha \eta_K -20160\alpha \eta_M }{60480}  t_j^6+ \\
    & \quad + 
    \frac{-420 \alpha^2 - 1680 \eta_K \alpha^2 + 280\alpha^3}{60480}  t_j^6 +
    \mathcal{O}(t_j^8).
\end{aligned}
\end{equation*} 
\end{remark}

With the choice of $\eta_K = \frac{31}{252}, \eta_M = \frac{23}{3780}$ and $\alpha = \frac{26}{21}$,
the stiffness reduction ratio and the asymptotic stiffness reduction ratio of GSFEMBQ are 
\begin{equation} 
\begin{aligned}
\rho_{gsq}(h) & = \frac{2 + \cos(\pi h) }{2 - \cos(\pi h) }
\cdot  \frac{95 + 31\cos(\pi h) }{95 - 31\cos(\pi h) }\cdot \frac{1179 - 688\cos(\pi h) + 23\cos(2\pi h)}{1179 + 688\cos(\pi h) + 23\cos(2\pi h) }, \\
\rho_{gsq,\infty} & = \lim_{h \to 0} \rho_{gsq}(h) = \frac{257}{160}.
\end{aligned}
\end{equation}
The asymptotic stiffness reduction percentage is $\varrho_{gsq,\infty} = \frac{9700}{257}$, meaning linear GSFEMBQ reduces the stiffness of Galerkin FEM by about $37.7\%$. 

\begin{remark}
    The choice of $\eta_K = \frac{31}{252}, \eta_M = \frac{23}{3780}$ and $\alpha = \frac{26}{21}$ leads to extra-superconvergent approximations when compared against than GSFEM and SoftFEMBQ.
    However, the stiffness reduction ratio is smaller than that of both GSFEM and SoftFEMBQ.
    We can further reduce the stiffness using a variable to optimize the reduction. Thus, we set 
    \begin{align}
    -\frac{1}{12} -\eta_K + \frac{\alpha}{6} & = 0, \\
    \frac{1 + 60\eta_K - 360\eta_M - 10 \alpha - 60 \alpha \eta_K + 10\alpha^2}{360} & = 0.
    \end{align}
    The solution is
    $$\eta_K = \frac{2\alpha - 1}{12}, \qquad \eta_M = \frac{5\alpha - 4}{360}.$$ 
    The case $\alpha=1$ reduces to GSFEM, see Section~\ref{sec:an1}, the case $\alpha=\frac{4}{5}$ reduces to the case of SoftFEMBQ, see Section~\ref{sec:an2}.
    This parameter allows us to obtain
    $$\rho_{gsq,\infty}  = \frac{37-20\alpha}{20 - 10\alpha},$$ 
    which is a decreasing function. Thus, the maximal reduction ratio is $\rho_{gsq,\infty}  = \frac{37}{20}$ (the corresponding percentage is $\varrho_{gsq,\infty} = 45.9\%$) when $\alpha=0$ (corresponding to the mass matrix integrated by the Gauss-Lobatto rule).
    In this case, $\eta_K = -\frac{1}{12}, \eta_M = -\frac{1}{90}.$
\end{remark}

\begin{remark}[Higher-order elements and multiple dimension problems]
    The quadratic and cubic elements for the standard Galerkin FEM were analyzed in \cite{DENG2021119}. 
    For softFEM and our proposed methods, the resulting matrices are much more complicated, and the analytical eigenpairs are unavailable. 
    Lastly, the analytical results for linear elements in 1D can be generalized to the case of tensor-product meshes in multiple dimensions by following, for example, \cite{Ainsworth10, calo2019dispersion, deng2021analytical}. 
\end{remark}

\section{Numerical Experiments}
\label{sec:numExperiments}

This section discusses numerical experiments to assess the performance of the methods proposed in Section~\ref{sec:main} and analyzed in Section~\ref{sec:ana}. These experiments examine the eigenvalue problems with various parameters, focusing on condition numbers and stiffness reduction ratios. Additionally, we validate the extra-superconvergence results established in Section~\ref{sec:ana}. 
These investigative efforts are instrumental in comparing the proposed methods against the standard Galerkin FEM and SoftFEM techniques, offering crucial insights into their respective benefits and efficiencies.

\subsection{Study on Superconvergence}
Firstly, we demonstrate the superconvergence results established in Section~\ref{sec:ana}.
For this purpose, we focus on the one-dimensional Laplace eigenvalue problem with $\Omega = (0, 1)$. The true eigenvalues and their corresponding $L^2$-normalized eigenfunctions are $ \lambda_j = j^2 \pi^2 \ \text{and} \ u_j(x) = \sqrt{2} \sin( j\pi x), j = 1, 2, \cdots.
$
The estimated eigenpairs are arranged in ascending order for comparison with their corresponding exact counterparts. To quantitatively analyze the differences, we present the relative eigenvalue errors, which are denoted as 
$\frac{\lambda_{*,j}^h - \lambda_j}{\lambda_j}$, where $\lambda_j$ denotes the $j$-th true (or reference) eigenvalue, and $\lambda_{*,j}^h$ is its counterpart obtained numerically. $*$ refers to the method; it is omitted when the context is clear. A smaller relative error indicates a closer alignment of our numerical solutions with the exact values, underscoring the accuracy of the computational approaches.
Table \ref{tab:superconvergence} shows the first eigenvalue relative errors when using linear GSFEM, SoftFEMBQ, GSFEMBQ with $\eta_K = \frac{31}{252}, \eta_M = \frac{23}{3780}$, $\alpha = \frac{26}{21}$, and GSFEMBQ with $\eta_K = -\frac{1}{12}, \eta_M = -\frac{1}{90}$, $\alpha = 0$. 
We used $N^h= 4,8,16$ and 32 uniform elements. 
The convergence rates are 6.03, 6.02, 8.4, and 5.97 respectively, confirming the superconvergence rates established in Section~\ref{sec:ana} when fixing the eigenvalue index $j=1$.
For linear elements, a superconvergence rate of 8 is first observed here.
This is 6 orders higher than the optimal convergence rate which is 2 for linear FEM.

\begin{table}[h!]
    \centering
    \footnotesize
    \begin{tabular}{ccccc}
    \toprule
    $N^h$ & $\frac{\lambda^h_{gs,1}-\lambda^h_1}{\lambda^h_1}$ & $\frac{\lambda^h_{sq,1}-\lambda^h_1}{\lambda^h_1}$ & $\frac{\lambda^h_{gsq,1}-\lambda^h_1}{\lambda^h_1}$ & $\frac{\hat\lambda^h_{gsq,1}-\lambda^h_1}{\lambda^h_1}$ \\
    \midrule

    $4$ & $4.22e{-5}$ & $7.41e{-5}$ & $2.58e{-6}$ & $1.90e{-4}$\\[1mm]
    $8$ & $6.20e{-7}$ & $1.13e{-6}$ & $9.56e{-9}$ & $3.10e{-6}$\\[1mm]
    $16$ & $9.53e{-9}$ & $1.75e{-8}$ & $3.68e{-11}$ & $4.91e{-8}$\\[1mm]
    $32$ & $1.48e{-10}$ & $2.73e{-10}$ & $6.58e{-14}$ & $7.69e{-10}$\\[1mm]
    Order & 6.03 & 6.02 & 8.4 & 5.97 \\[1mm]   
    \bottomrule
    \end{tabular}
    \caption{Eigenvalue errors for the first eigenvalue when using GSFEM, SoftFEMBQ, GSFEMBQ ($\lambda^h_{gsq,1}, \eta_K = \frac{31}{252}, \eta_M = \frac{23}{3780}$, $\alpha = \frac{26}{21}$), and GSFEMBQ ($\hat\lambda^h_{gsq,1}, \eta_K = -\frac{1}{12}, \eta_M = -\frac{1}{90}$, $\alpha = 0$).}
    \label{tab:superconvergence}
\end{table}

For larger eigenvalues, Figure \ref{fig:superconvergence} shows the whole spectra in logarithmic scale when using $N^h=32$ uniform elements. 
The superconvergence rates are 6.18, 6.02, 8.37, and 5.76 for GSFEM, SoftFEMBQ, GSFEMBQ with $\eta_K = \frac{31}{252}, \eta_M = \frac{23}{3780}$, $\alpha = \frac{26}{21}$, and GSFEMBQ with $\eta_K = -\frac{1}{12}, \eta_M = -\frac{1}{90}$, $\alpha = 0$, respectively. 
This confirms the superconvergence rates in terms of index $j$ established in Section~\ref{sec:ana} when fixing mesh size $h=1/N^h$.

\begin{figure}[h!]
    \centering
    \includegraphics[width=0.4\textwidth]{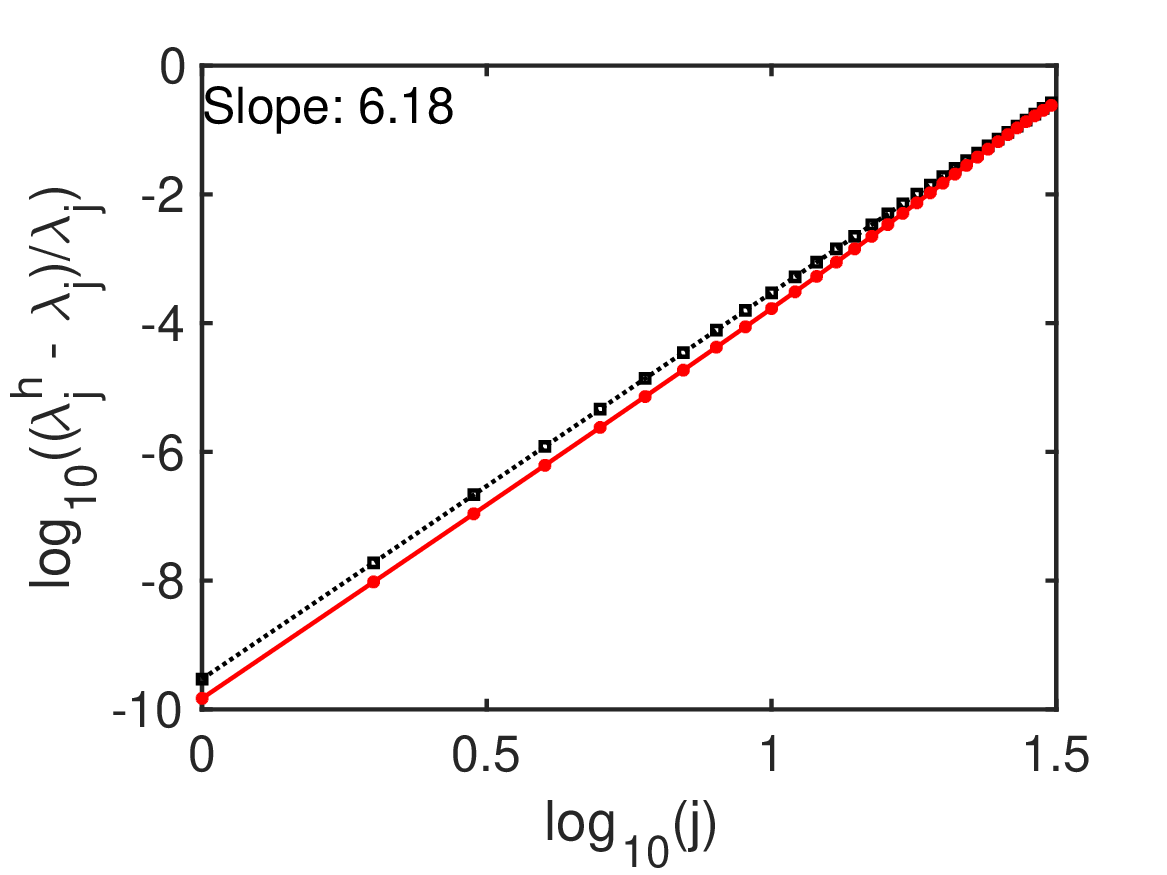}
    \includegraphics[width=0.4\textwidth]{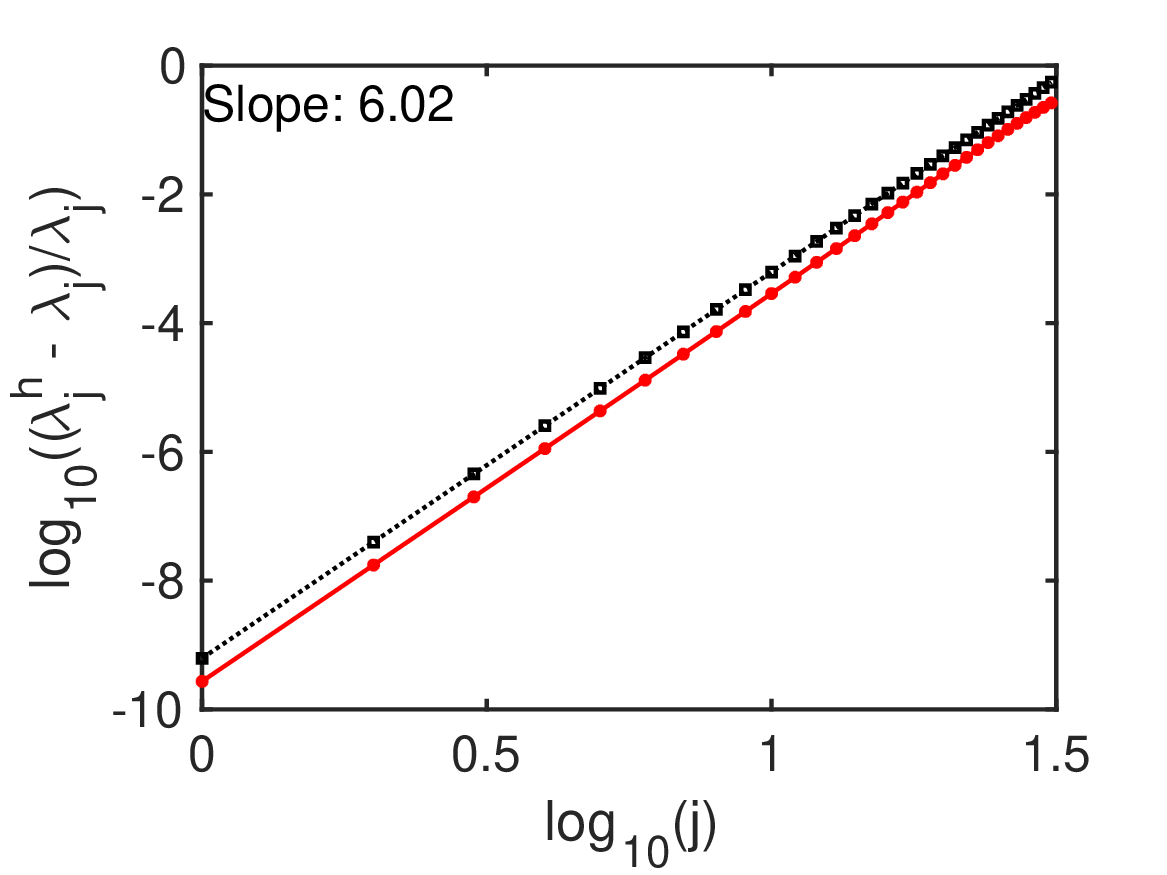}
    \includegraphics[width=0.4\textwidth]{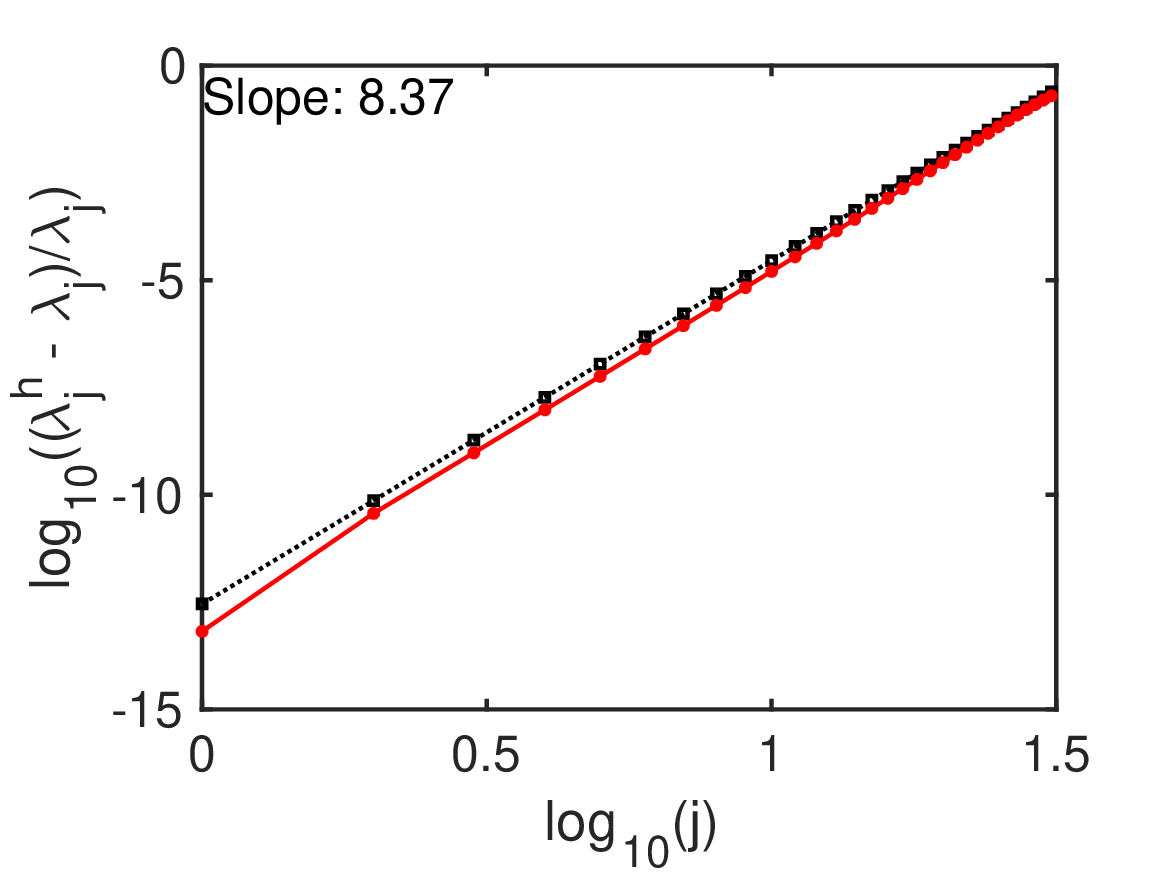}
    \includegraphics[width=0.4\textwidth]{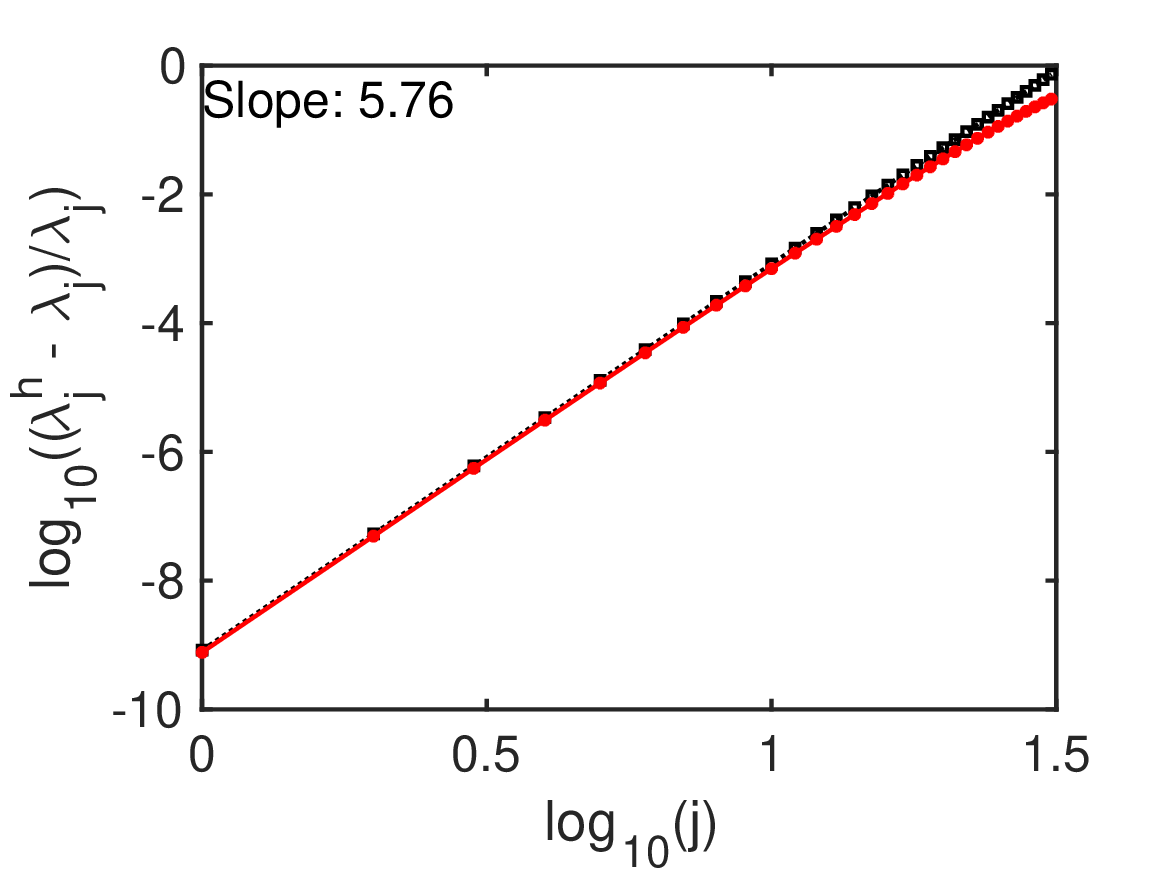}

    \caption{Eigenvalue superconvergence (red lines) for various methods with linear uniform elements with size $N^{h}=32$. 
    The black dotted lines are reference lines with theoretic orders (see Section \ref{sec:ana}) 6, 6, 8, and 6, respectively, for GSFEM (top left), SoftFEMBQ (top right), GSFEMBQ with $\eta_K = \frac{31}{252}, \eta_M = \frac{23}{3780}$ and $\alpha = \frac{26}{21}$ (bottom left), and GSFEMBQ with $\eta_K = -\frac{1}{12}, \eta_M = -\frac{1}{90}$ and $\alpha = 0$ (bottom right).}
    \label{fig:superconvergence}
\end{figure}

\subsection{Study on Stiffness Reduction}

We now focus on the reduction of the stiffness of the discretized systems, characterized by condition numbers. 
Our investigation commences with the study of GSFEM. Subsequently, we study the approach of SoftFEMBQ and GSFEMBQ  with a comprehensive comparison of all other methods under consideration.

\subsubsection{GSFEM}
\label{sec:MainResultsGSFEM}

\begin{table}[h!]
    \centering
    \begin{tabular}{cccc}
    \toprule
    $p$ & $\eta_{K}$ & $\eta_{M,\max}$ & Selected $\eta_{M}$ \\
    \midrule
    $1$ & $\frac{1}{12}$ & $\frac{5}{144}$ & $\frac{1}{360}$ \\[1mm]
    $2$ & $\frac{1}{24}$ & $\frac{1}{1439}$ & $\frac{1}{2880}$\\[1mm]
    $3$ & $\frac{1}{40}$ & $\frac{1}{28270}$ & $\frac{1}{57600}$\\
    \bottomrule
    \end{tabular}
    \caption{Softness parameters $\eta_{M}$ with $\eta_{K}=\frac{1}{2(p+1)(p+2)}$ for $p=1, 2, 3$.}
    \label{tab:result1}
\end{table}

We first determine an upper bound for the softness parameter $\eta_{M}$ based on the monotonicity of eigenvalues. 
We first note that the coercivity of the bilinear form $b_{gs} (\cdot,\cdot)$ is guaranteed as long as $\eta_{M}\ge 0$. 
With the choice $\eta_{K} = \frac{1}{2(p+1)(p+2)}$ as in the softFEM 
and the insight of Theorem~\ref{thm:gsfemp1}, 
the choice of the parameter $\eta_{M}$ is summarized in Table~\ref{tab:result1}. 
Therein, $\eta_{M, \max}$ shows the maximal values of $\eta_{M}$ such that the sorted eigenvalues of the discretized system are paired with energy-increasing eigenfunctions (for larger $\eta_{M}>\eta_{M, \max}$, the energies of eigenfunctions oscillate at high-frequency region after sorting the eigenvalues).
In particular, for linear elements, $\eta_{M, \max} = \frac{5}{144}$ is derived analytically using \eqref{eq:evm1} so that the discrete eigenvalues are non-decreasing with respect to the index in \eqref{eq:evm1}. 
In other words, the function $\lambda_{gs,j}^h = \lambda_{gs}^h(j)$ is monotonically increasing only if $\eta_m \le \eta_{M, \max} = \frac{5}{144}$ for linear elements. 
For higher-order elements, the values of $\eta_{M, \max}$ are obtained numerically.

\begin{figure}
    \centering
    \begin{subfigure}[b]{0.32\textwidth}
    \includegraphics[width=\textwidth]{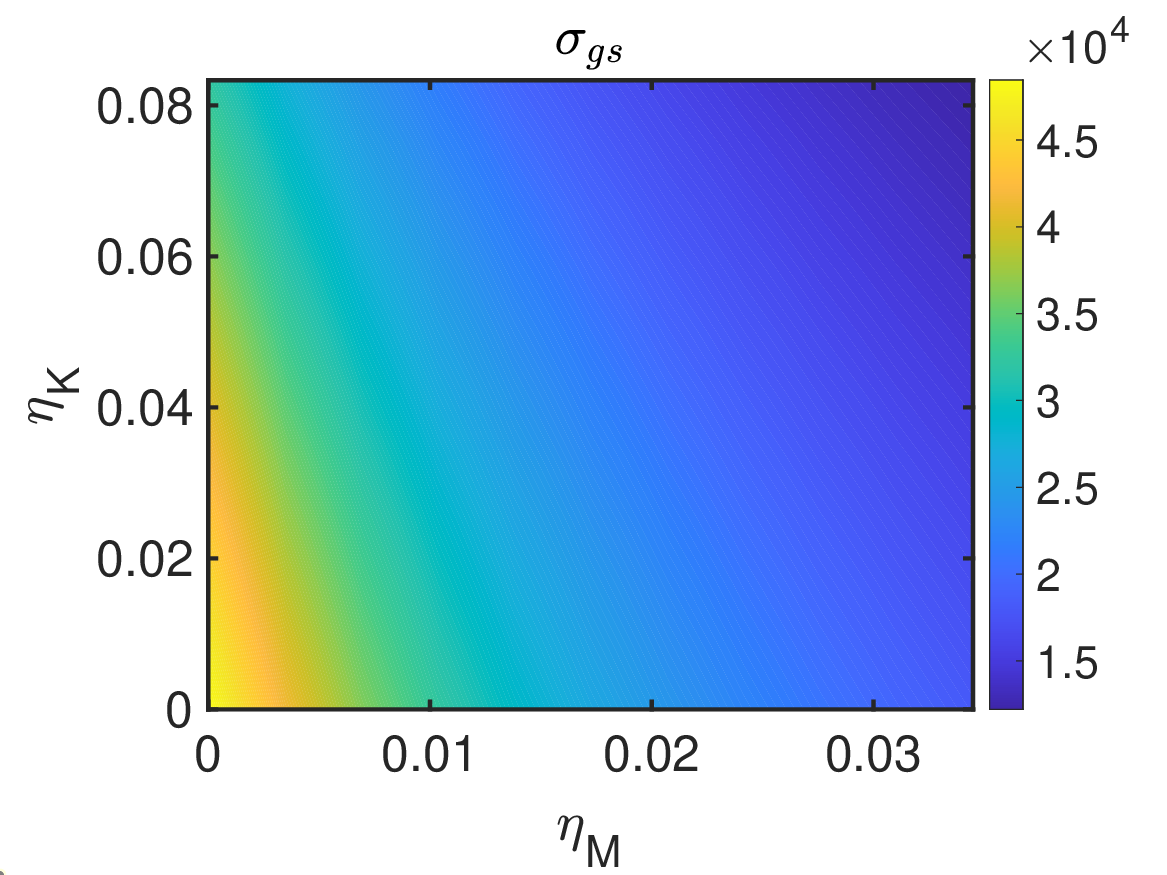}
    \end{subfigure}
    \begin{subfigure}[b]{0.32\textwidth}
    \includegraphics[width=\textwidth]{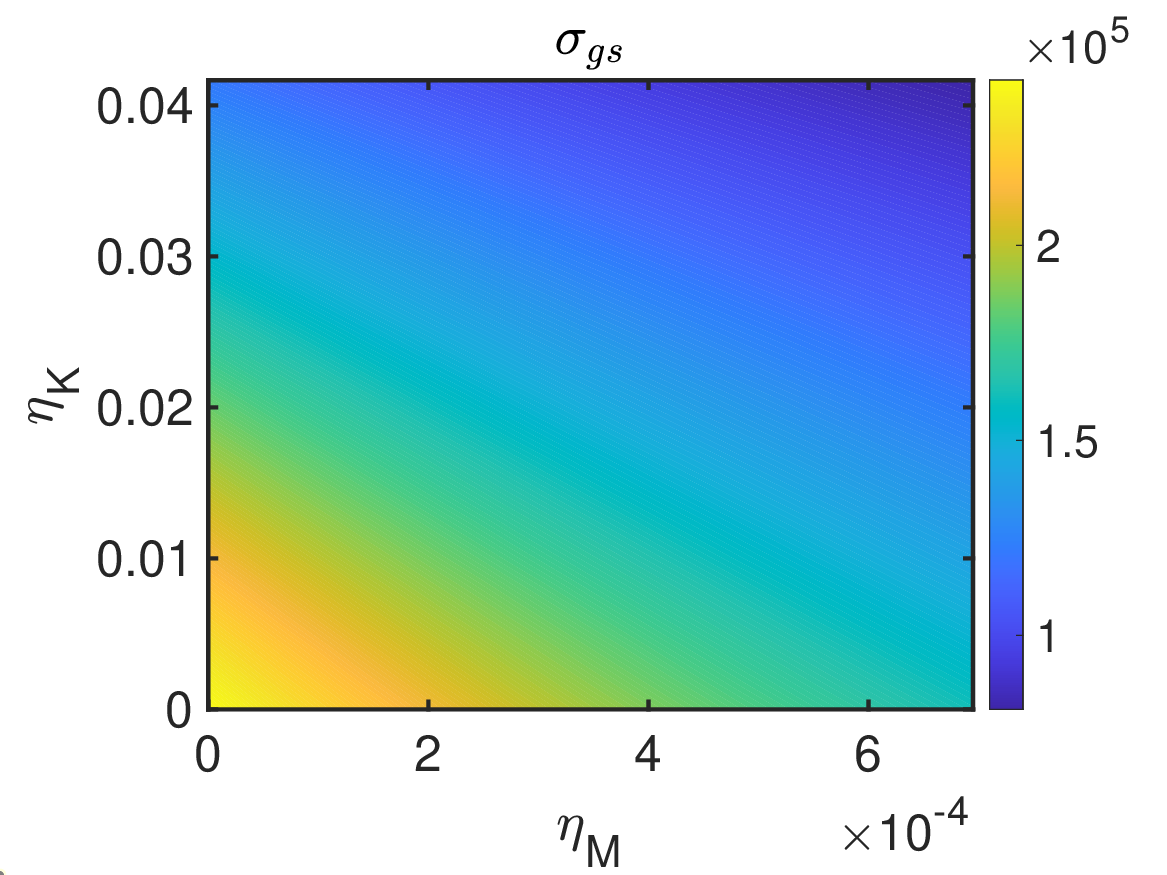}
    \end{subfigure}
    \begin{subfigure}[b]{0.32\textwidth}
    \includegraphics[width=\textwidth]{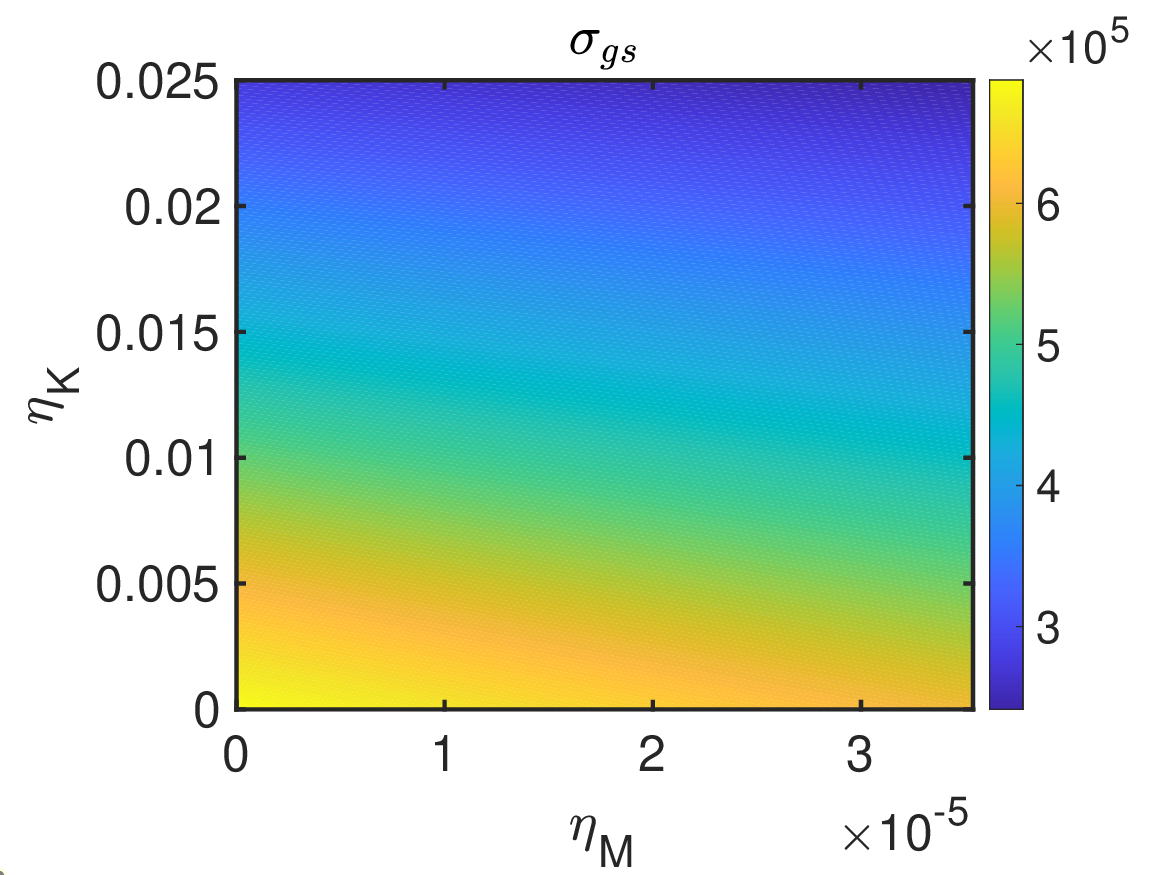}
    \end{subfigure}
    \caption{Condition number $\sigma_{gs}$ of 1D Laplace eigenvalue problem when using GSFEM with $N^{h}=200$ uniform elements with $p=1$ (left),  $p=2$ (middle),  and $p=3$ (right). }
    \label{fig:GSFEM_comparisionKM}
\end{figure}

Using these parameters in Table~\ref{tab:result1}, Figure~\ref{fig:GSFEM_comparisionEV&E} illustrates a key finding of the GSFEM. 
For linear elements, the eigenvectors of the matrix eigenvalue problem of GSFEM are the same as the eigenvectors of that of FEM; see the analytical eigenvectors in \eqref{eq:evm1}. 
Thus, the eigenfunction error of GSFEM maintains the same as the Galerkin FEM and SoftFEM. 
For quadratic and cubic elements, the GSFEM demonstrates lower eigenfunction errors compared to those observed in Galerkin FEM and SoftFEM. In our experimental scenarios, selecting a higher value of the softness parameter $\eta_{M}$ leads to a more significant decrease in eigenfunction errors. 
Moreover, for a mesh with $N^h=200$ uniform elements, as the softness parameter $\eta_{M}$ increases, there is a consistent decrease in the condition number, as demonstrated in Figure~\ref{fig:GSFEM_comparisionKM}. With $\eta_{K}$ and $\eta_{M,\max}$ in Table~\ref{tab:result1}, the stiffness reduction ratios of the GSFEM are approximately $3.98$, $3.00$, and $2.86$ for $p=1$, $p=2$, and $p=3$, respectively. In comparison, for SoftFEM, the reduction rates are nearly $1.50$, $2.00$, and $2.50$, respectively. 
Thus, GSFEM further reduces the condition numbers of SoftFEM. 
This is also expected in the view of Rayleigh quotient $\lambda^h_* = \frac{a_s(u^h, u^h) }{b_{*}(u^h, u^h)}$ as GSFEM tends to have a larger denominator than SoftFEM. 

We now consider the 2D Laplace eigenvalue problems on tensor-product meshes. We study the spectral problem in \eqref{eq:pde} with $\Omega=(0,1)^2$ and $\kappa=1$. The true eigenvalues and normalized eigenfunctions are
$
\lambda_{ij} = (i^2 + j^2)\pi^2
$
and 
$u_{ij}(x,y) = 2 \sin(i\pi x)\sin(j\pi y)$, $i,j = 1, 2, \cdots$, respectively.

\begin{figure}[h!]
\centering
\includegraphics[width=\textwidth]{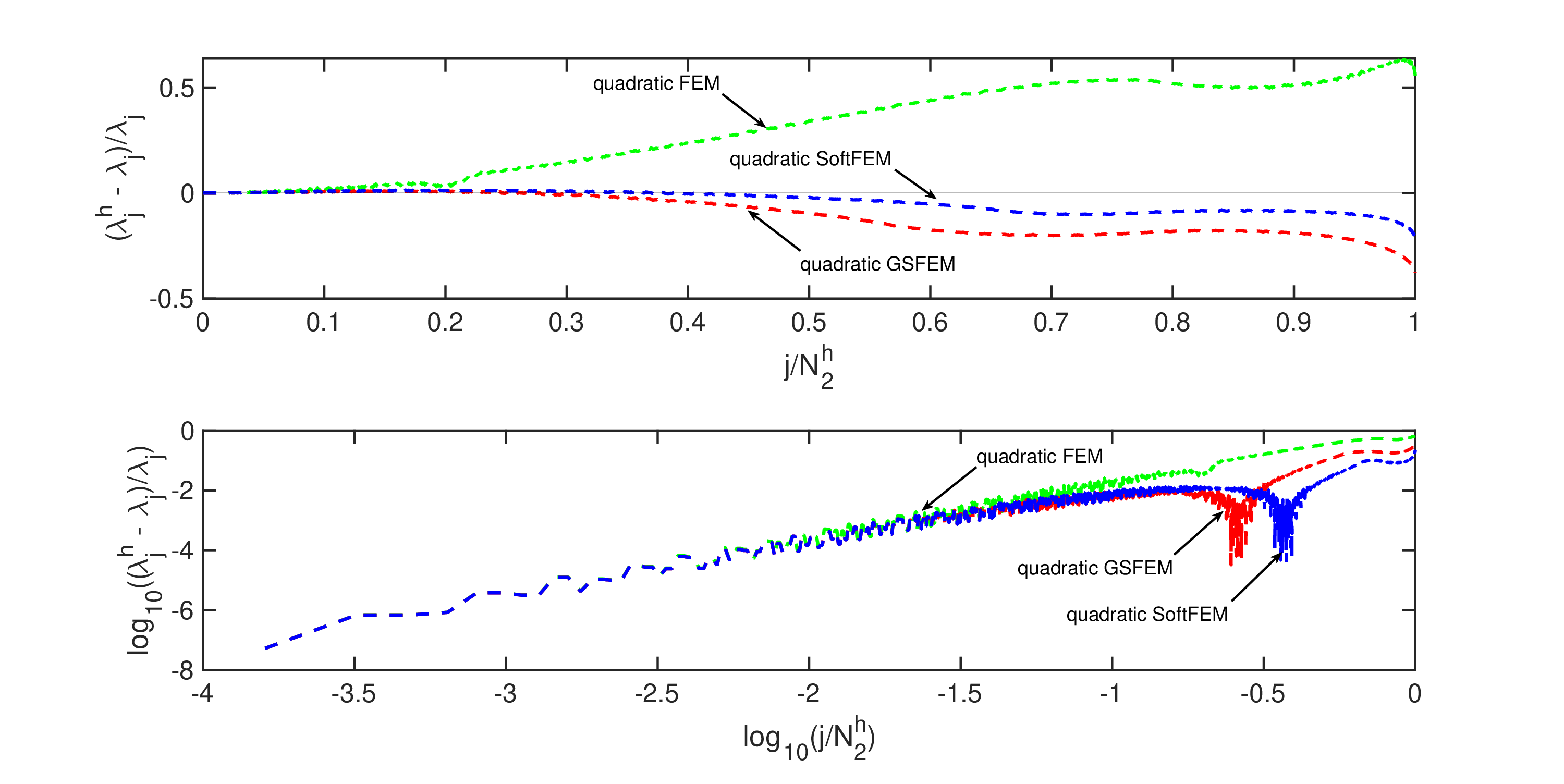}
\caption{Relative eigenvalue errors for the 2D Laplace eigenvalue problem when using quadratic Galerkin FEM, SoftFEM and GSFEM with a mesh of size $40\times 40$.}
\label{fig:GSFEM2dp2}
\includegraphics[width=\textwidth]{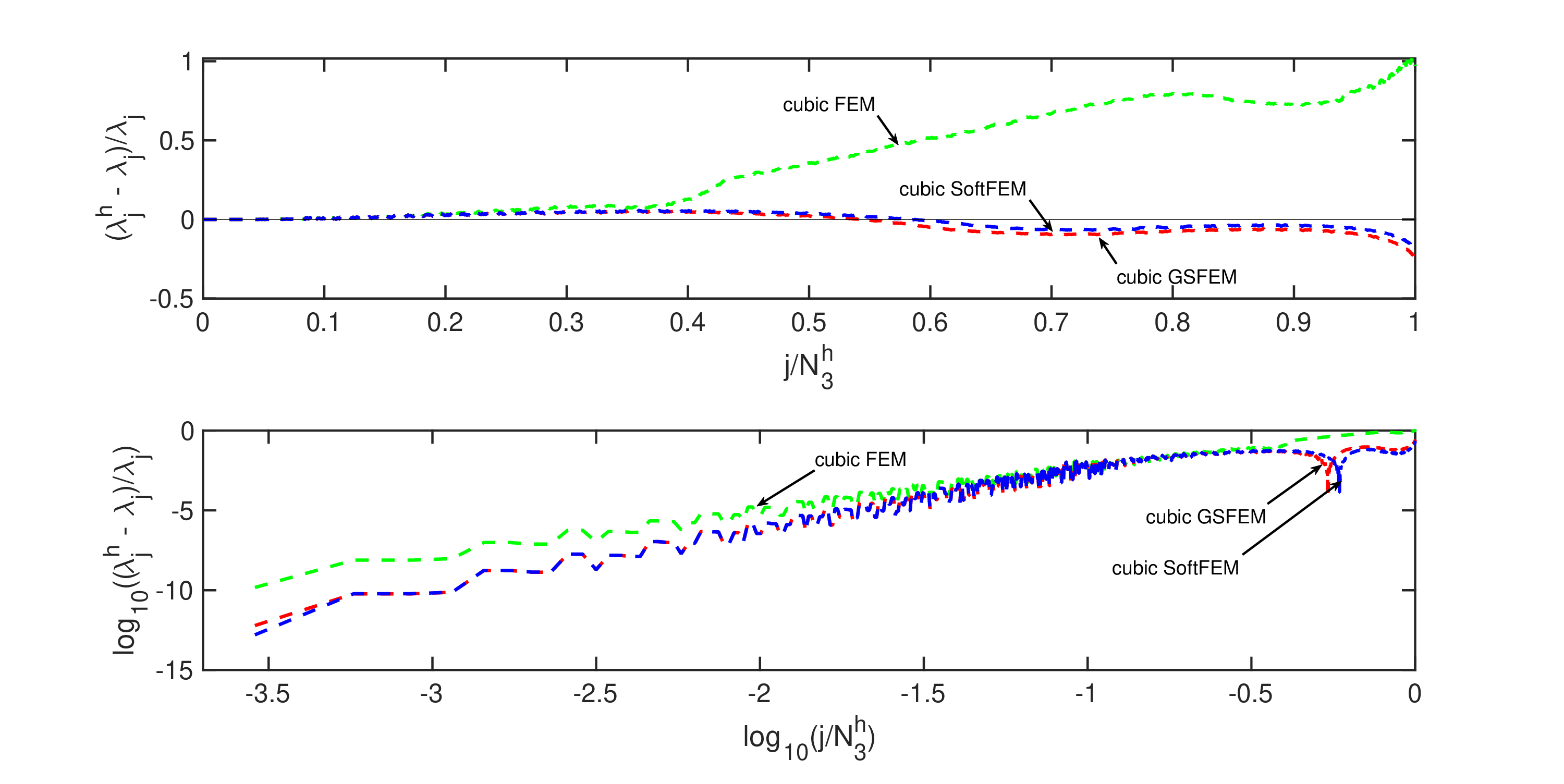}
\caption{Relative eigenvalue errors for the 2D Laplace eigenvalue problem when using cubic Galerkin FEM, SoftFEM and GSFEM with a mesh of size $20\times 20$.}
\label{fig:GSFEM2dp3}
\end{figure}

Inspired by the outcomes of the one-dimensional numerical experiments discussed before, we adopt the parameter setting $\eta_{K}=\frac{1}{2(p+1)(p+2)}$ with $\eta_{M}$ in Table~\ref{tab:result1}.
Figures~\ref{fig:GSFEM2dp2} and~\ref{fig:GSFEM2dp3} depict the relative errors in eigenvalues observed when applying quadratic and cubic Galerkin FEM, SoftFEM and GSFEM in the two-dimensional setting. For quadratic elements, a uniform mesh with $40\times40$ elements is utilized, while for cubic elements, a uniform mesh of size $20\times20$ is used. 
We observe that GSFEM further reduces the stiffness of the softFEM discretized system.

\subsubsection{Generalization with Blending Quadratures and Comparisons} 
\label{sec:MainResultsSoftFEMBQ}

Following Section~\ref{sec:MainResultsGSFEM}, we now study the generalization with blending quadratures. 
Similarly, we consider the 1D Laplace eigenvalue problem with the domain $\Omega = (0, 1)$. We use various blending parameters to demonstrate the performance of the method.
Herein, we also study the eigenfunction errors and their comparisons. We employ the $L^2$ norm and the $H^1$ semi-norm. The $L^2$ norm error measures the difference between the numerically computed eigenfunctions and the reference or exact solutions. In contrast, the $H^1$ semi-norm error characterizes the accuracy in the gradient of the eigenfunctions.
The $L^2$ norm error is quantified as $\lVert u_j - u_j^h \rVert_{L^2}$, and the $H^1$ semi-norm error is normalized and defined as $\frac{|u_j - u_j^h|_{H^1}}{\lambda_j}$. 
The normalization of the $H^1$ error by the eigenvalue is crucial as it renders the error measure more meaningful relative to the magnitude of the eigenvalue. These error assessments are pivotal in evaluating the efficacy of the numerical methods in accurately capturing the system's modal behaviour, especially in high-order spectral regions.

\begin{figure}[h!]
    \centering

    \includegraphics[width=0.45\textwidth]{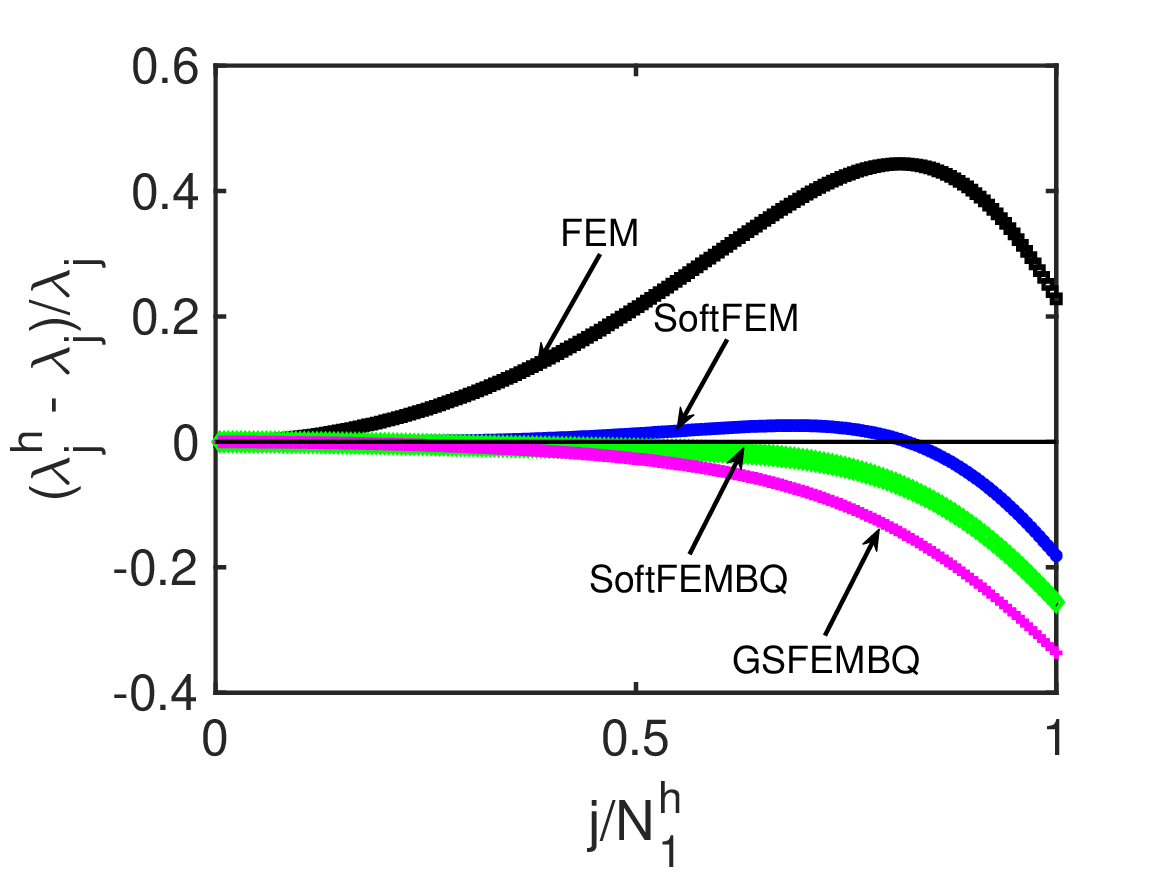}
    \includegraphics[width=0.45\textwidth]{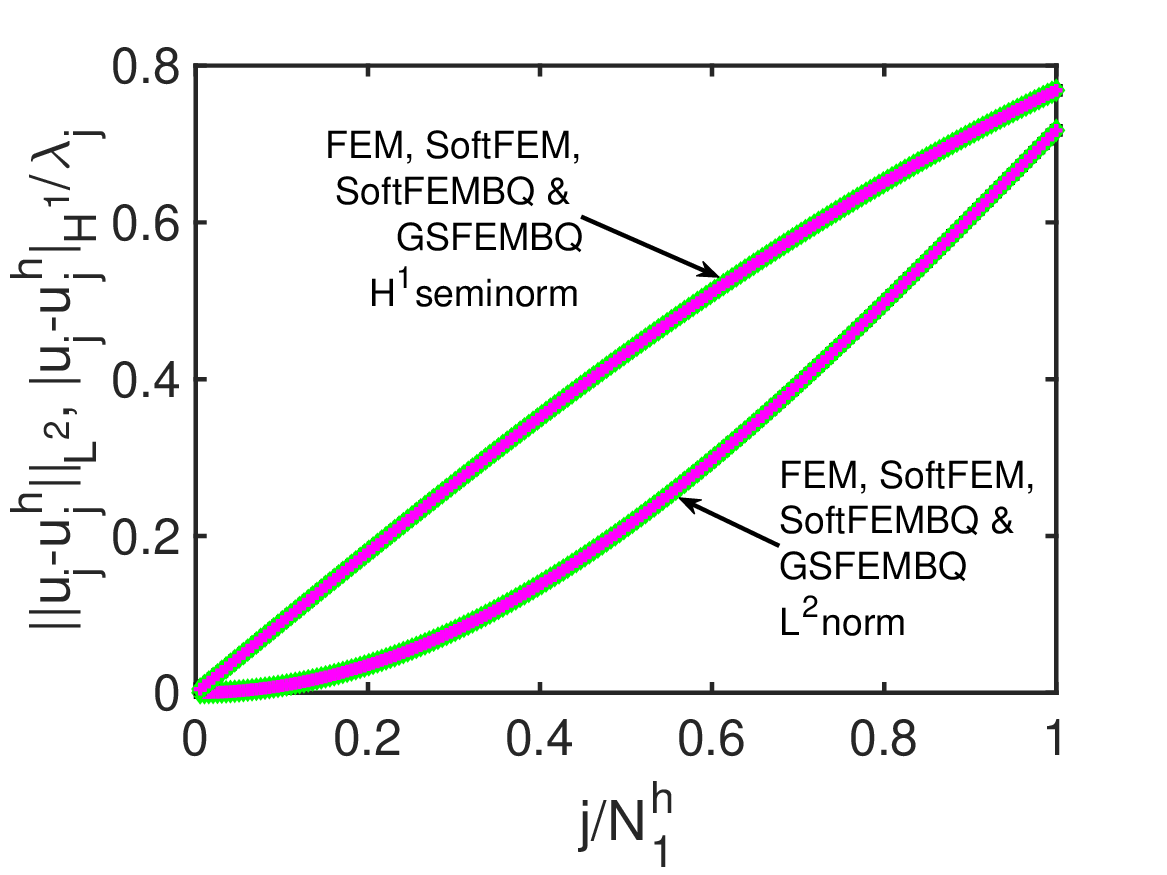}
    \includegraphics[width=0.45\textwidth]{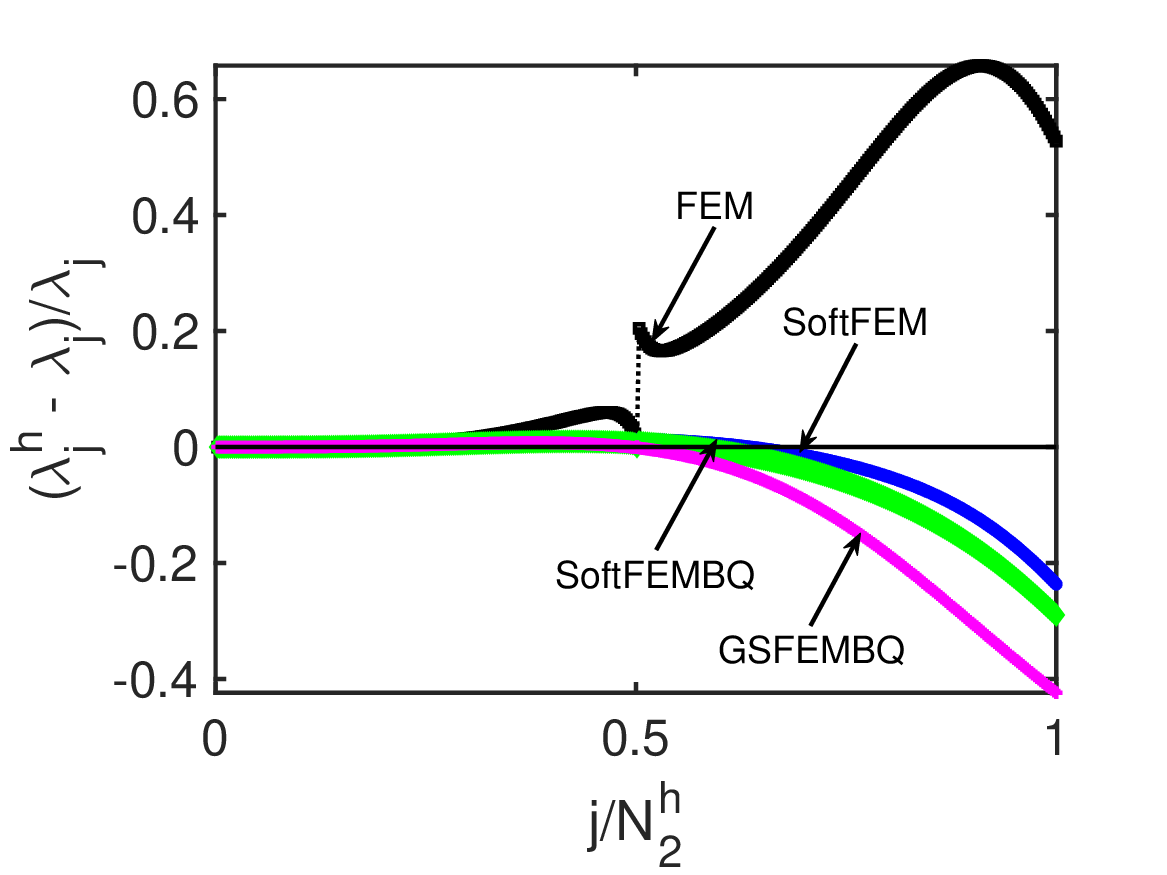}
    \includegraphics[width=0.45\textwidth]{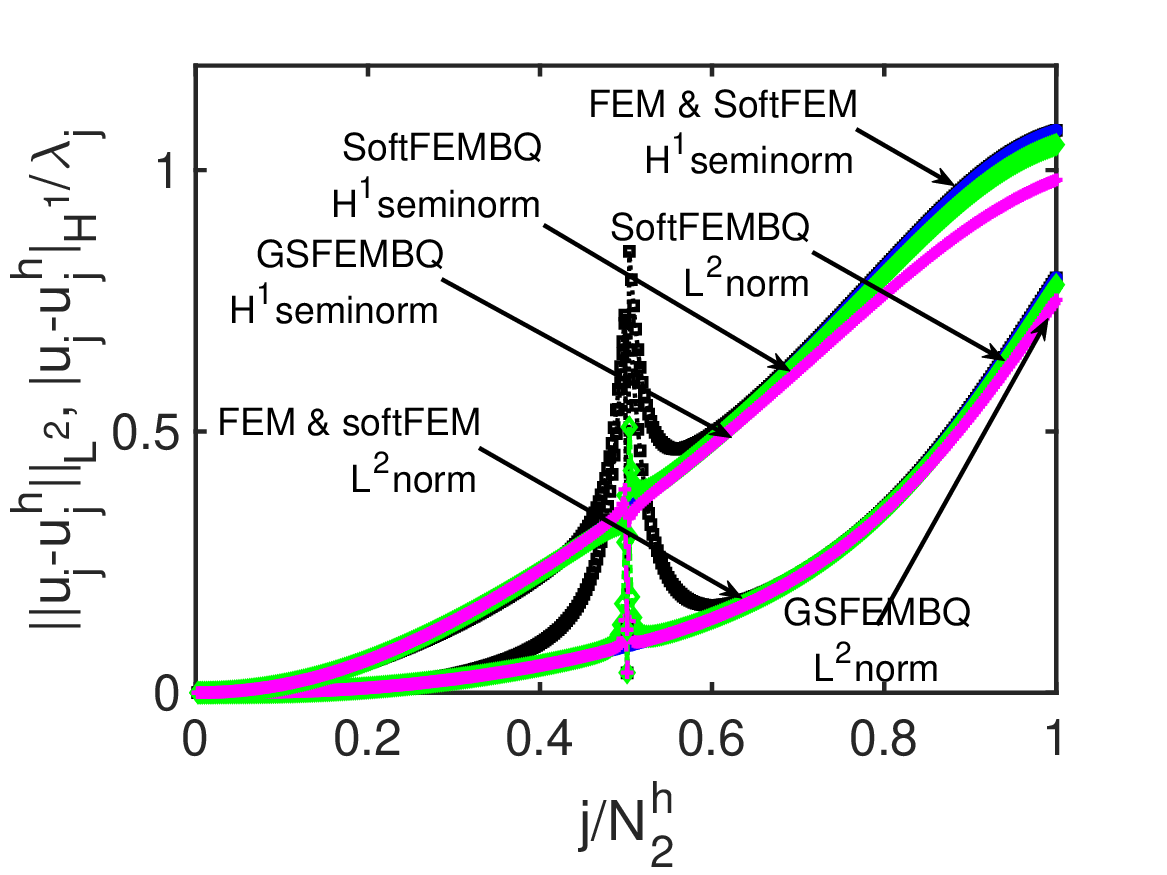}
    \includegraphics[width=0.45\textwidth]{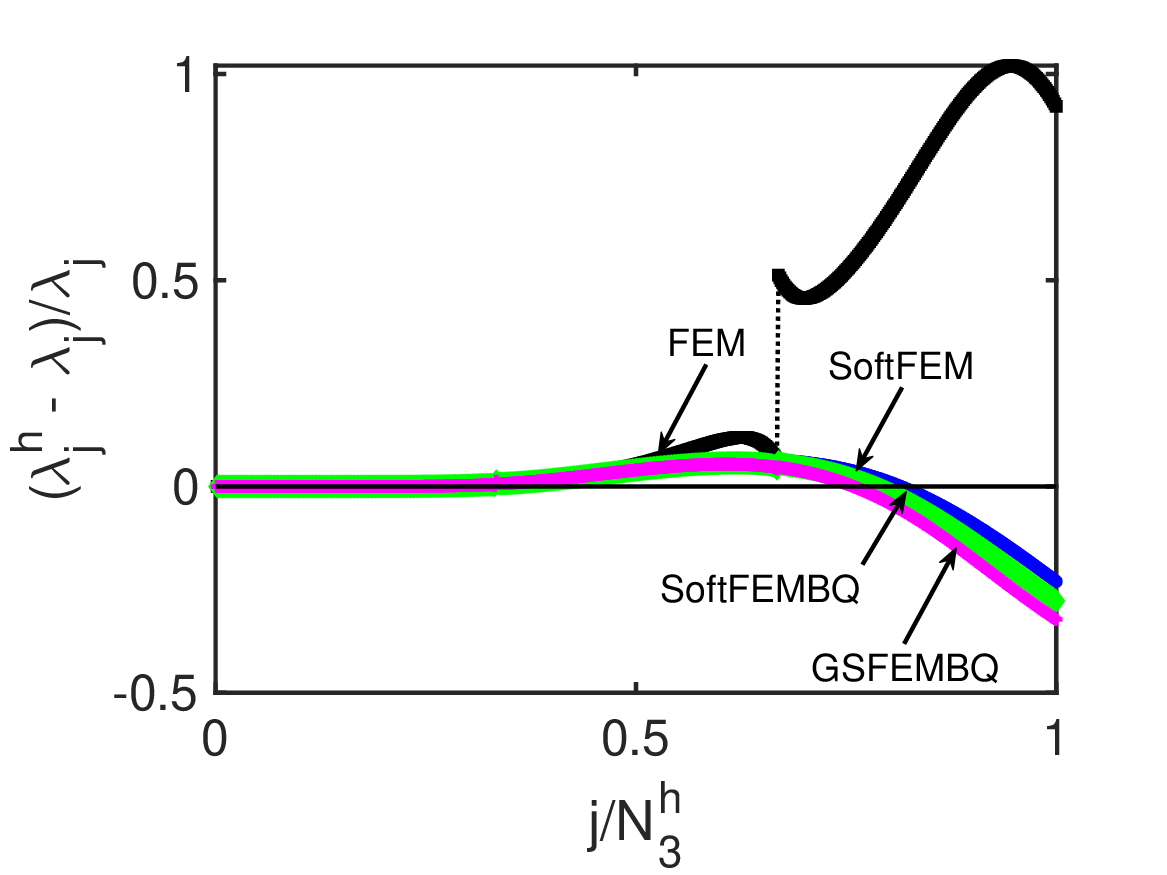}
    \includegraphics[width=0.45\textwidth]{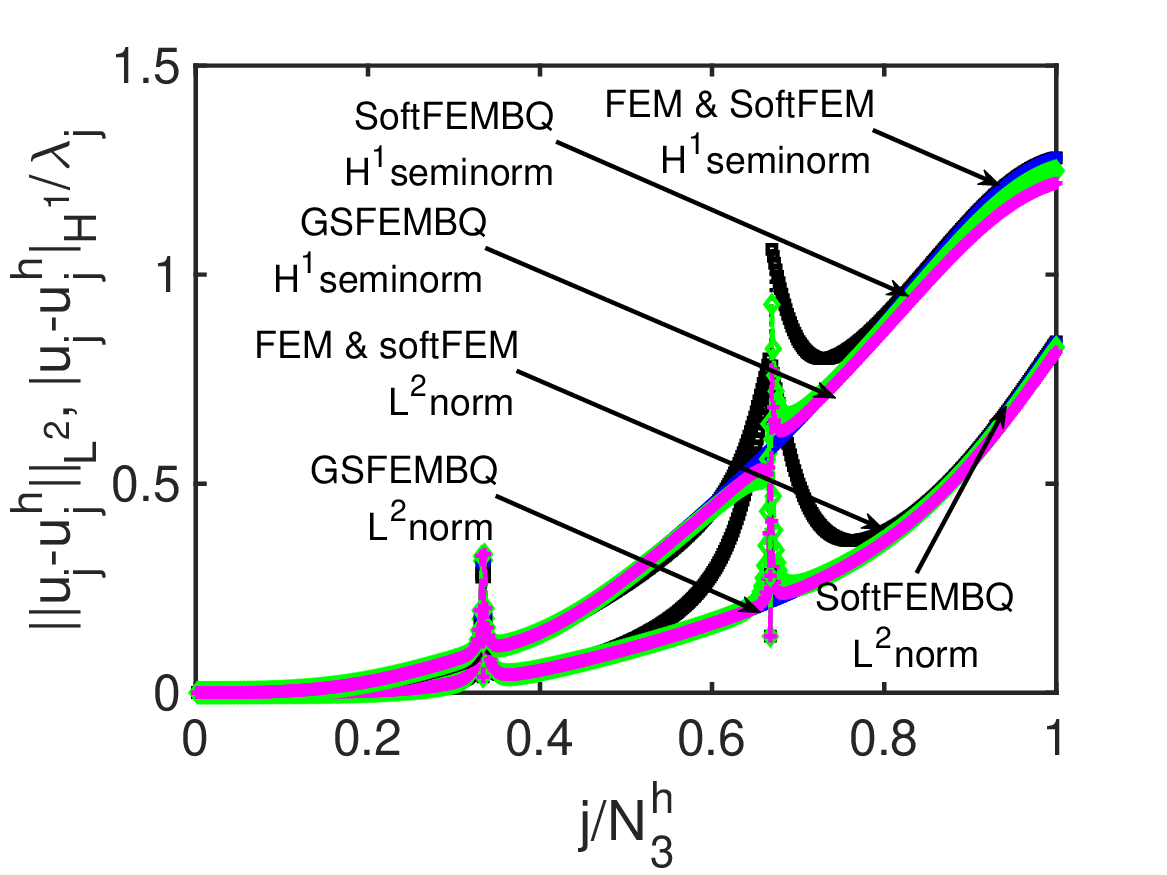}

    \caption{Comparison of eigenvalue errors (left) and eigenfunction errors (right) using different methods for Laplacian eigenvalue problem in 1D with $N^h=200$ uniform elements: Galerkin FEM (black), SoftFEM (blue), SoftFEMBQ (green), and GSFEMBQ (magenta). The parameters are chosen as in Table~\ref{tab:result1} and the blending parameter $\alpha=0.95$. }
    \label{fig:softFEMBQ_comparisionEV&E}
\end{figure}

Figure~\ref{fig:softFEMBQ_comparisionEV&E} shows the comparison of eigenvalue and eigenfunction errors while using Galerkin FEM, SoftFEM, SoftFEMBQ, and GSFEMBQ. Therein, the blending parameter is $\alpha=0.95$. The other parameters are shown in Table ~\ref{tab:result1}.
We observe lower stiffness/condition numbers while using blending quadratures.


With the parameters in Table~\ref{tab:result1}, Table~\ref{tab:softFEMBQ_conditionnumber} shows the advantages of the proposed methods on the reduction of condition numbers compared to the standard Galerkin FEM and SoftFEM. 
Both the SoftFEMBQ and GSFEMBQ methods lead to a further reduction in the condition number when compared to SoftFEM.
Among these methods, we observe that GSFEMBQ achieves the most significant reduction in condition number compared to other methods. 
The blending quadrature plays an important role in stiffness reduction.

\begin{table}[h!]
    \centering
    \footnotesize
    \begin{tabular}{C{0cm}C{0.4cm}C{0.8cm}C{0.8cm}C{0.8cm}C{0.8cm}C{0.8cm}C{0.8cm}C{0.8cm}C{0.8cm}C{0.8cm}C{0.8cm}C{0.8cm}}
    \toprule
    $p$ & $\alpha$ & $\lambda_{max}^h$ & $\lambda_{s,max}^h$ & $\lambda_{gs,max}^h$ & $\lambda_{sq,max}^h$ & $\lambda_{gsq,max}^h$ & $\sigma$ & $\sigma_{s}$ & $\sigma_{gs}$ & $\sigma_{sq}$ & $\sigma_{gsq}$ & $\rho_{gsq}$ \\
    \midrule

    $1$ & $0.00$ & $4.80e5$ & $3.20e5$ & $2.82e5$ &$1.07e5$ & $1.02e5$ & $4.86e4$ & $3.24e4$ & $2.86e4$ & $1.08e4$ & $1.03e4$ & $4.72$\\[1mm]
    & $0.10$ & $4.80e5$ & $3.20e5$ & $2.82e5$ &$1.14e5$ & $1.10e5$ &$4.86e4$ & $3.24e4$ & $2.86e4$ & $1.16e4$ & $1.11e4$ & $4.38$\\[1mm]
    & $0.30$ & $4.80e5$ & $3.20e5$ & $2.82e5$ &$1.33e5$ & $1.26e5$ & $4.86e4$ & $3.24e4$ & $2.86e4$  & $1.35e4$ & $1.28e4$ & $3.80$\\[1mm]
    & $0.50$ & $4.80e5$ & $3.20e5$ & $2.82e5$ &$1.60e5$ & $1.50e5$ & $4.86e4$ & $3.24e4$ & $2.86e4$   &$1.62e4$ & $1.52e4$ & $3.20$\\[1mm]
    & $0.70$ & $4.80e5$ & $3.20e5$ & $2.82e5$ &$2.00e5$ & $1.85e5$ & $4.86e4$ & $3.24e4$ & $2.86e4$ & $2.03e4$ & $1.87e4$ & $2.60$\\[1mm]
    & $0.95$ & $4.80e5$ & $3.20e5$ & $2.82e5$ &$2.91e5$ & $2.59e5$ & $4.86e4$ & $3.24e4$ & $2.86e4$ & $2.95e4$ & $2.63e4$ & $1.85$\\[1mm]
    
    $2$ & $0.78$ & $2.40e6$ & $1.20e6$ & $9.60e5$ &$9.00e5$ & $7.58e5$ & $2.43e5$ & $1.22e5$ & $9.73e4$ & $9.12e4$ & $7.68e4$ & $3.17$\\[1mm]
    & $0.80$ & $2.40e6$ & $1.20e6$ & $9.60e5$ &$9.23e5$ & $7.74e5$ & $2.43e5$ & $1.22e5$ & $9.73e4$ & $9.35e4$ & $7.84e4$ & $3.10$\\[1mm]
    & $0.95$ & $2.40e6$ & $1.20e6$ & $9.60e5$ &$1.12e6$ & $9.06e5$ & $2.43e5$ & $1.22e5$ & $9.73e4$ & $1.13e5$ & $9.18e4$ & $2.66$\\[1mm]
    
    $3$ & $0.94$ & $6.80e6$ & $2.73e6$ & $2.55e6$ &$2.53e6$ & $2.37e6$ & $6.89e5$ & $2.76e5$ & $2.58e5$ & $2.56e5$ & $2.40e5$ & $2.87$\\[1mm]
    & $0.95$ & $6.80e6$ & $2.73e6$ & $2.55e6$ &$2.56e6$ & $2.40e6$ & $6.89e5$ & $2.76e5$ & $2.58e5$ & $2.59e5$ & $2.43e5$ & $2.84$\\
    
    \bottomrule
    \end{tabular}
    \caption{Comparison of key metrics - maximum eigenvalue, condition number, and stiffness reduction ratios - among Galerkin FEM, SoftFEM, GSFEM, SoftFEMBQ, and GSFEMBQ methods with $p=1, 2, 3$ and a mesh of $N^h=200$. The parameters are in Table~\ref{tab:result1}. }
    \label{tab:softFEMBQ_conditionnumber}
\end{table}

We now consider a case with blending parameters discovered in Ainsworth et al. \cite{Ainsworth10}. This strategy, which entails the variation of $\eta_{K}$ and $\eta_{M}$, yields a distinct set of numerical outcomes. However, the condition number and eigenfunction error trends remain largely similar to previous findings. In more details, the blending parameter $\alpha = \frac{1}{p+1}$, and the soft parameter $\eta_{K}$ and $\eta_{M}$ as in Table~\ref{tab:result2}. We observe that the eigenfunction errors for methods with blended quadratures are lower than those in FEM and SoftFEM, particularly in the higher end of the spectrum; see Figure~\ref{fig:softFEMBQ_comparisionEV&E2}.
Furthermore, using the parameters in Table~\ref{tab:result2}, Table \ref{tab:softFEMBQ_conditionnumber2} shows similar stiffness reduction compared to the previous case when using parameters in Table~\ref{tab:result1}.


\begin{table}[h!]
    \centering
    \begin{tabular}{c|c|c|c}
    \toprule
    $p$ & $\alpha$ & Selected $\eta_{K}$ & Selected $\eta_{M}$ \\
    \midrule
    $1$ & $\frac{1}{2}$ & $\frac{1}{8}$ & $\frac{1}{96}$ \\[1mm]
    $2$ & $\frac{1}{3}$ & $\frac{1}{32}$ & $\frac{1}{3840}$\\[1mm]
    $3$ & $\frac{1}{4}$ & $\frac{1}{72}$ & $\frac{1}{84480}$\\
    \bottomrule
    \end{tabular}
    \caption{Softness parameter $\eta_{K}$ and $\eta_{M}$ with $\alpha = \frac{1}{p+1}$, for $p=1, 2, 3$.}
    \label{tab:result2}
\end{table}

\begin{figure}[h!]
    \centering

    \includegraphics[width=0.45\textwidth]{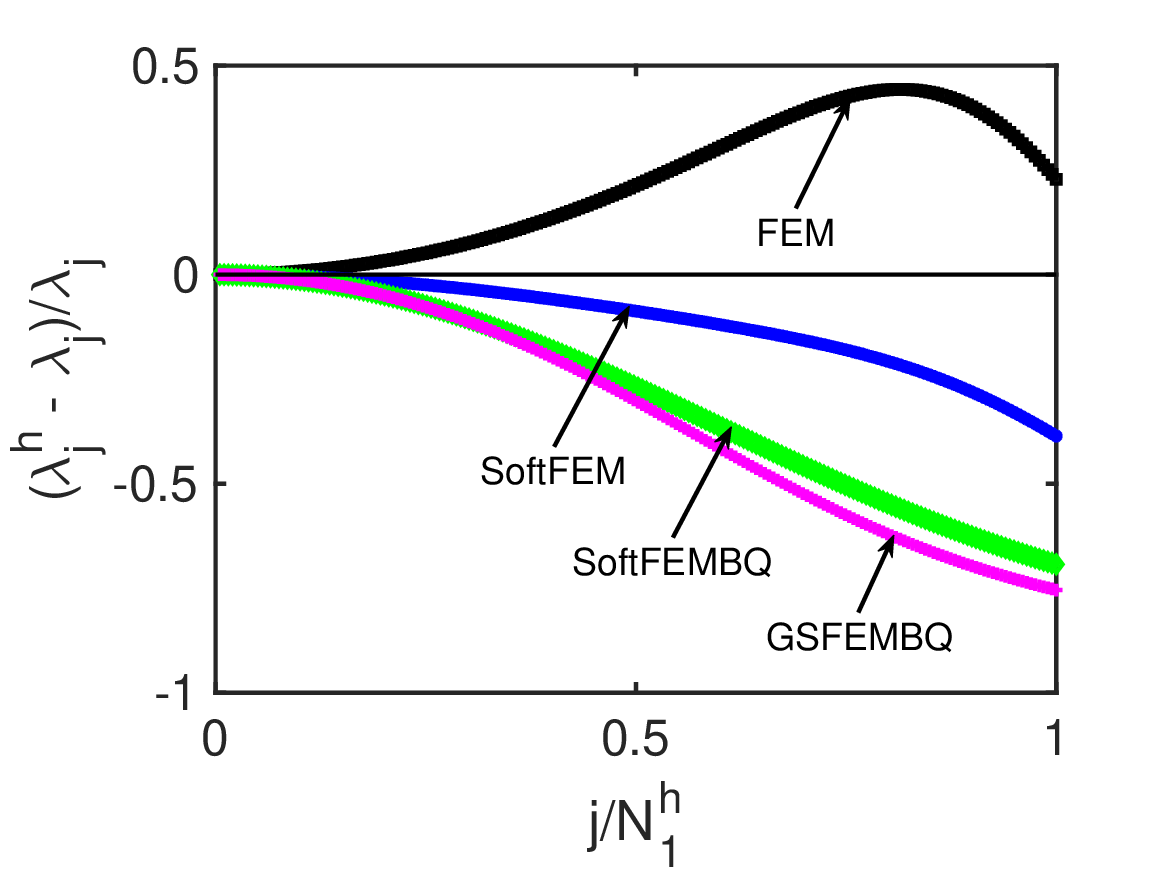}
    \includegraphics[width=0.45\textwidth]{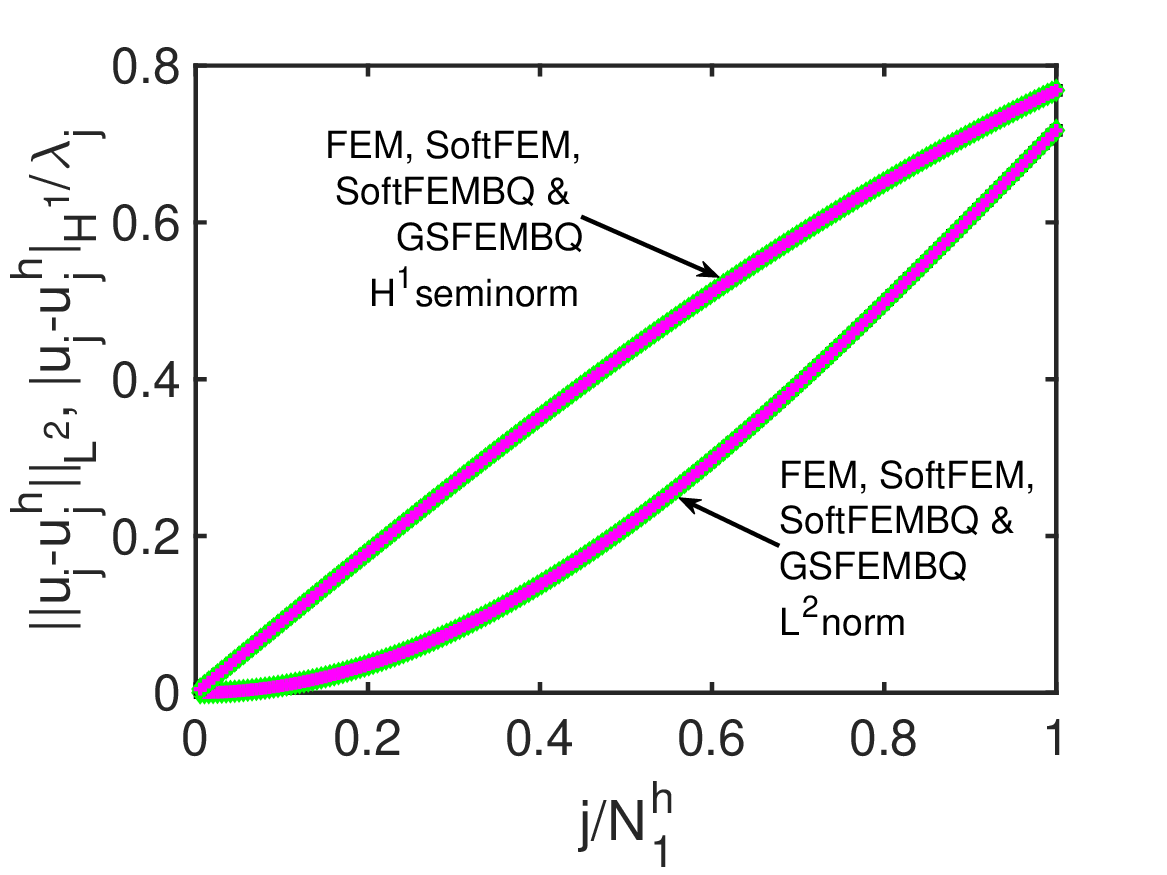}
    \includegraphics[width=0.45\textwidth]{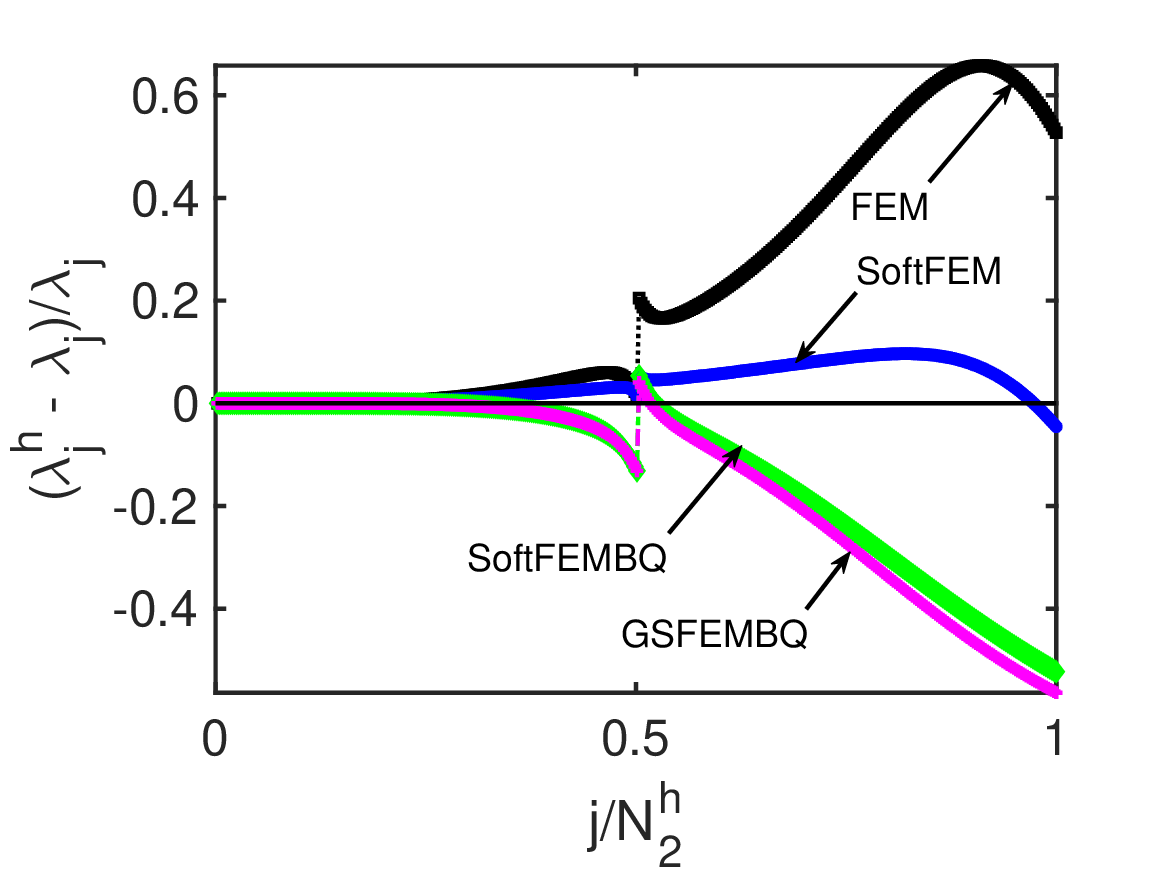}
    \includegraphics[width=0.45\textwidth]{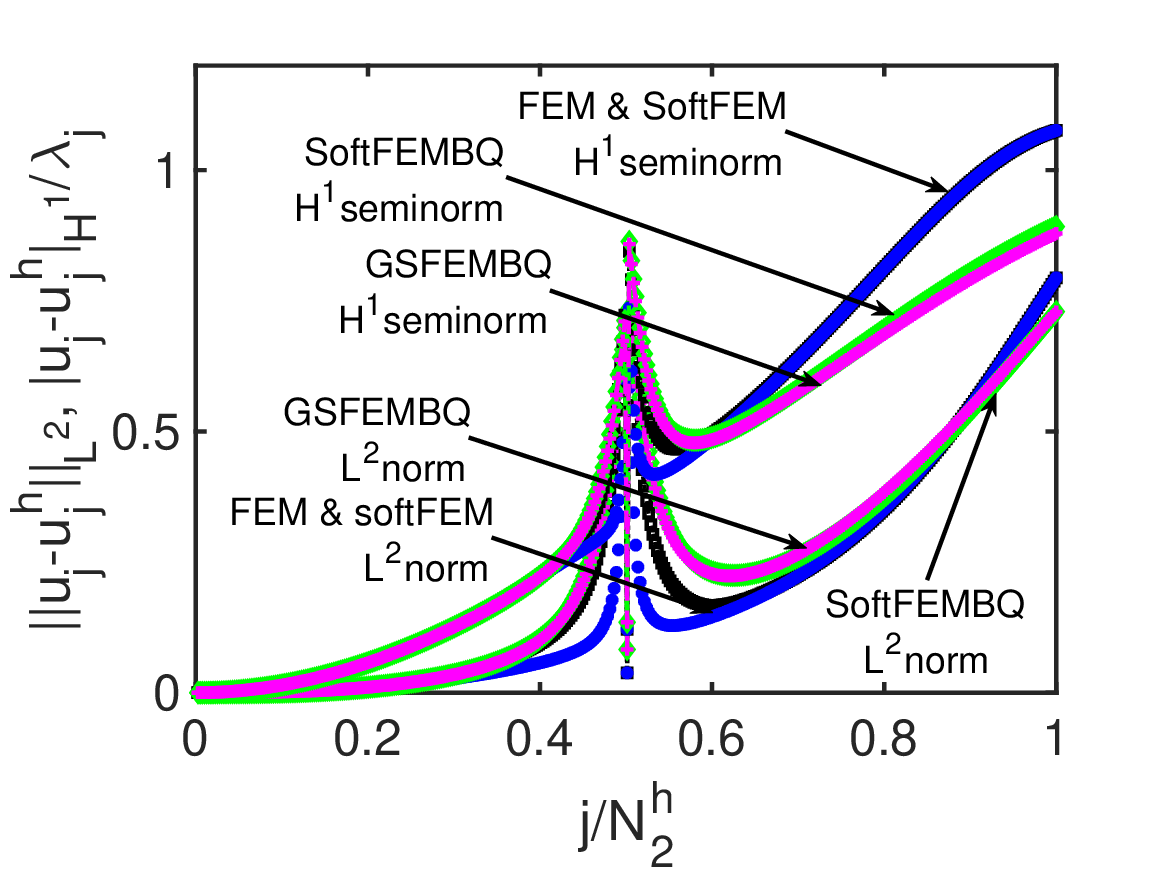}
    \includegraphics[width=0.45\textwidth]{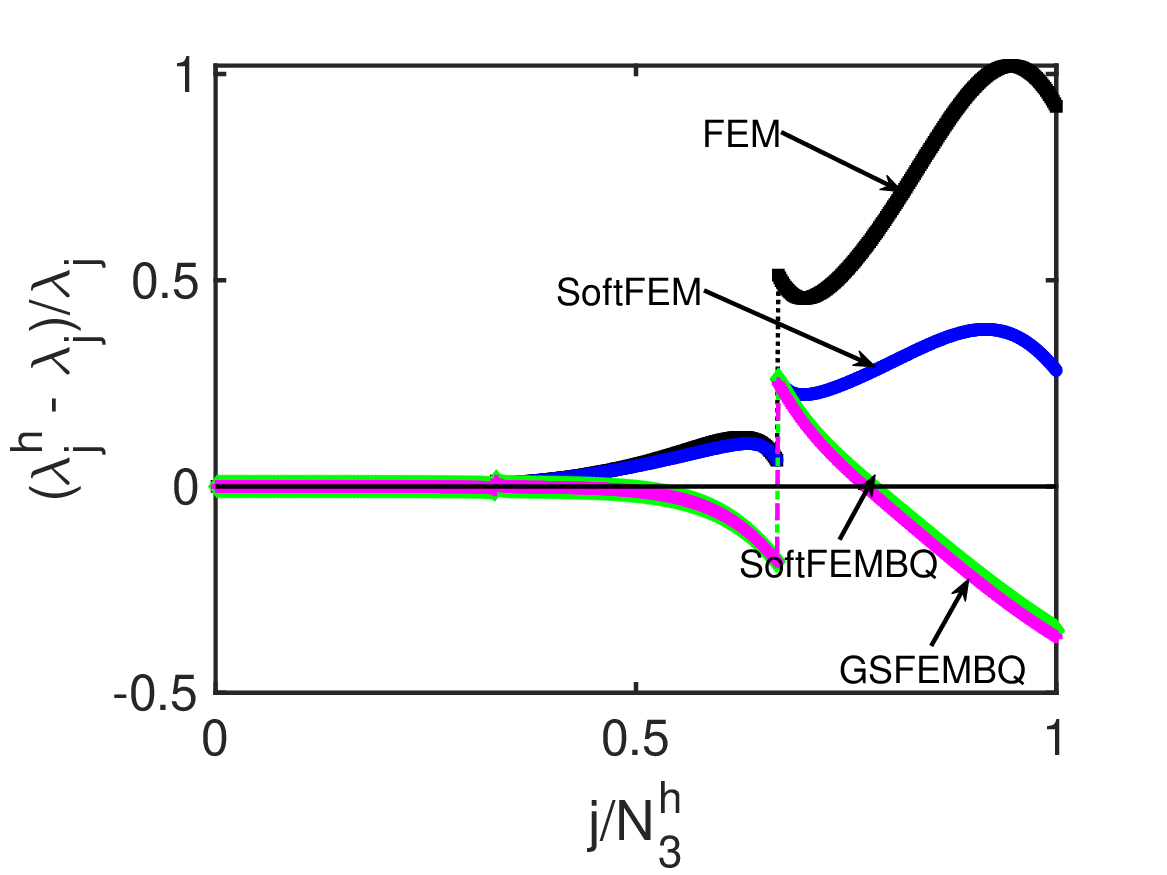}
    \includegraphics[width=0.45\textwidth]{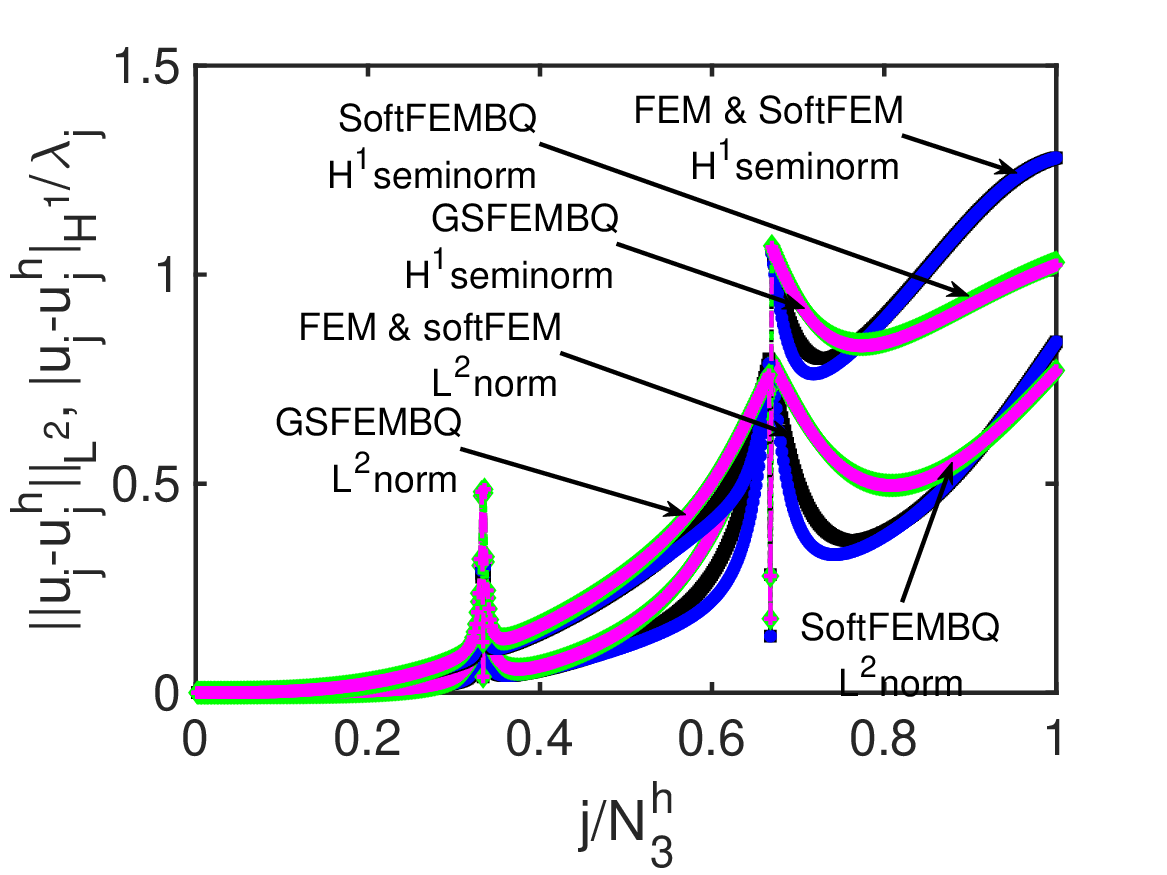}

    \caption{Comparison of eigenvalue errors (left) and eigenfunction errors (right) using different methods for Laplacian eigenvalue problem in 1D with $N^h=200$ uniform elements: Galerkin FEM (black), SoftFEM (blue), SoftFEMBQ (green), and GSFEMBQ (magenta). The parameters are chosen as in Table~\ref{tab:result2}. } 
\label{fig:softFEMBQ_comparisionEV&E2}
\end{figure}



\begin{table}[h!]
    \centering
    \footnotesize
    \begin{tabular}{C{0cm}C{0.4cm}C{0.8cm}C{0.8cm}C{0.8cm}C{0.8cm}C{0.8cm}C{0.8cm}C{0.8cm}C{0.8cm}C{0.8cm}C{0.8cm}C{0.8cm}}
    \toprule
    $p$ & $\alpha$ & $\lambda_{max}^h$ & $\lambda_{s,max}^h$ & $\lambda_{gs,max}^h$ & $\lambda_{sq,max}^h$ & $\lambda_{gsq,max}^h$ & $\sigma$ & $\sigma_{s}$ & $\sigma_{gs}$ & $\sigma_{sq}$ & $\sigma_{gsq}$ & $\rho_{gsq}$ \\
    \midrule

    $1$ & $0.50$ & $4.80e5$ & $2.40e5$ & $1.60e5$ &$1.20e5$ & $9.60e4$ & $4.86e4$ & $2.43e4$ & $1.62e4$ & $1.22e4$ & $9.73e3$ & $5.00$\\[1mm]

    $2$ & $0.33$ & $2.40e6$ & $1.50e6$ &$1.26e6$ &$7.50e5$ & $6.86e5$ & $2.43e5$ & $1.52e5$ & $1.28e5$ & $7.60e4$ & $6.95e4$ & $3.50$\\[1mm]
    
    $3$ & $0.25$ & $6.80e6$ & $4.54e6$ &$4.33e6$ &$2.30e6$ & $2.25e6$ & $6.89e5$ &$4.60e5$ & $4.39e5$ & $2.33e5$ & $2.28e5$ &$3.02$\\
    
    \bottomrule
    \end{tabular}
    \caption{Comparison of key metrics - maximum eigenvalue, condition number, and stiffness reduction ratios - among Galerkin FEM, SoftFEM, GSFEM, SoftFEMBQ, and GSFEMBQ methods with $p=1, 2, 3$ and a mesh of $N^h=200$. The parameters are in Table~\ref{tab:result2}.}
    \label{tab:softFEMBQ_conditionnumber2}
\end{table}



    
    

Tables~\ref{tab:softFEMBQ_stiffness reduction ratios1} and~\ref{tab:softFEMBQ_stiffness reduction ratios2} show the comparison of stiffness reduction ratios for different methods with $p=1, 2, 3$ and a uniform mesh of size $N^h=200$. The parameters used in these experiments are in Tables~\ref{tab:result1} ($\alpha=0.95$) and~\ref{tab:result2}, respectively. In both Tables, we observed that our new methods exhibit higher stiffness reduction ratios than the SoftFEM, indicating lower condition numbers. In Table~\ref{tab:softFEMBQ_stiffness reduction ratios1}, we observe that, for all methods, as the order of polynomials increases, the stiffness reduction ratio also increases. Whereas in Table~\ref{tab:softFEMBQ_stiffness reduction ratios2}, we observe a decrease but reduction ratios are larger than SoftFEM. Notably, the stiffness reduction ratio of SoftFEMBQ is almost twice that of SoftFEM.
The choices of the parameters have a major impact on the stiffness reduction of the discretized systems.

\begin{table}[h!]
    \centering
    \footnotesize
    \begin{tabular}{ccccc}
    \toprule
    $p$ & $\rho_{s}$ & $\rho_{gs}$ & $\rho_{sq}$ & $\rho_{gsq}$ \\
    \midrule

    $1$ & $1.50$ & $1.70$ & $1.65$ & $1.85$\\[1mm]

    $2$ & $2.00$ & $2.51$ & $2.15$ &$2.66$\\[1mm]
    
    $3$ & $2.50$ & $2.67$ & $2.66$ &$2.84$\\
    
    \bottomrule
    \end{tabular}
    \caption{Comparison of stiffness reduction ratios for SoftFEM, GSFEM, SoftFEMBQ, and GSFEMBQ methods with  $p=1, 2, 3$ and a mesh of size $N^h=200$. Parameters are in Table~\ref{tab:result1} and the blending parameter $\alpha=0.95$.}
    \label{tab:softFEMBQ_stiffness reduction ratios1}
\end{table}

\begin{table}[h!]
    \centering
    \footnotesize
    \begin{tabular}{ccccc}
    \toprule
    $p$ & $\rho_{s}$ & $\rho_{gs}$ & $\rho_{sq}$ & $\rho_{gsq}$ \\
    \midrule

    $1$ & $2.00$ & $3.00$ & $4.00$ & $5.00$\\[1mm]

    $2$ & $1.60$ & $1.90$ & $3.20$ &$3.50$\\[1mm]
    
    $3$ & $1.50$ & $1.57$ & $2.95$ &$3.02$\\
    
    \bottomrule
    \end{tabular}
    \caption{
    Comparison of stiffness reduction ratios for SoftFEM, GSFEM, SoftFEMBQ, and GSFEMBQ methods with  $p=1, 2, 3$ and a mesh of size $N^h=200$. Parameters are in Table~\ref{tab:result2}.}
    \label{tab:softFEMBQ_stiffness reduction ratios2}
\end{table}

\subsection{An Example with Variable Diffusion Coefficient}
Lastly, in this subsection, we consider a study for the problem \eqref{eq:pde} with variable diffusion coefficient. 
For simplicity, we focus on the one-dimensional case with $\Omega=(0,1)$ and a diffusion coefficient $\kappa=e^{x+x^2}$.
In this case, the exact eigenvalues are not available. 
In order to quantify the errors, one often uses fine mesh with high-order elements to generate highly accurate eigenvalues and use them as reference eigenvalues.
Herein, we use quartic elements with a fine mesh of size $N^h=1000$ to generate the reference eigenvalues. 
Figure \ref{fig:dc_1} shows the comparison of the relative eigenvalue errors when using different methods with linear, quadratic, and cubic elements.   
Table \ref{tab:dc_1_1} shows the comparison of eigenvalues, condition numbers, and stiffness reduction ratios. We observe similar performance of the methods studied in previous subsections.

\begin{figure}
    \centering
    \hspace{-0.5cm}
    \includegraphics[width=0.32\textwidth]{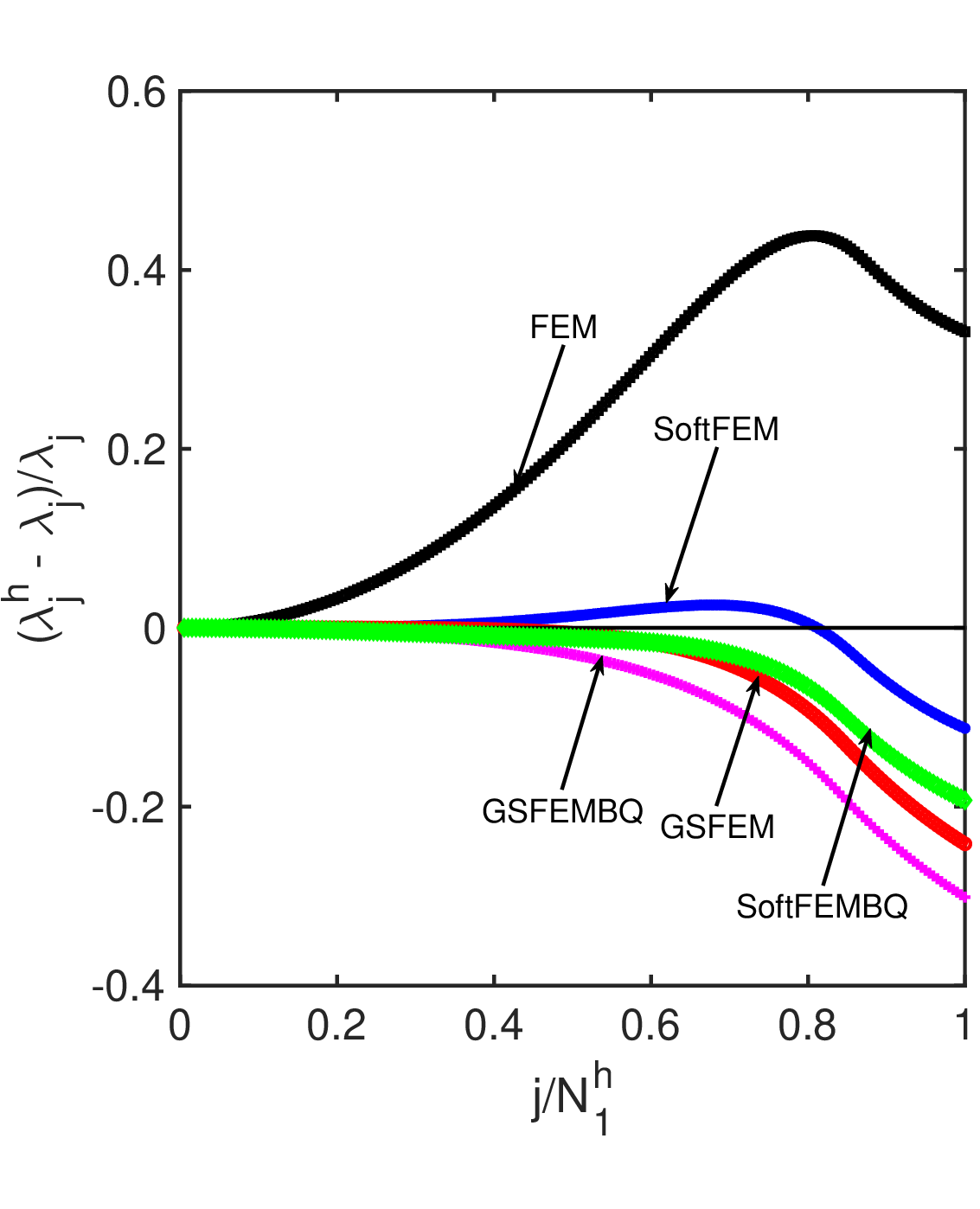}
    \includegraphics[width=0.32\textwidth]{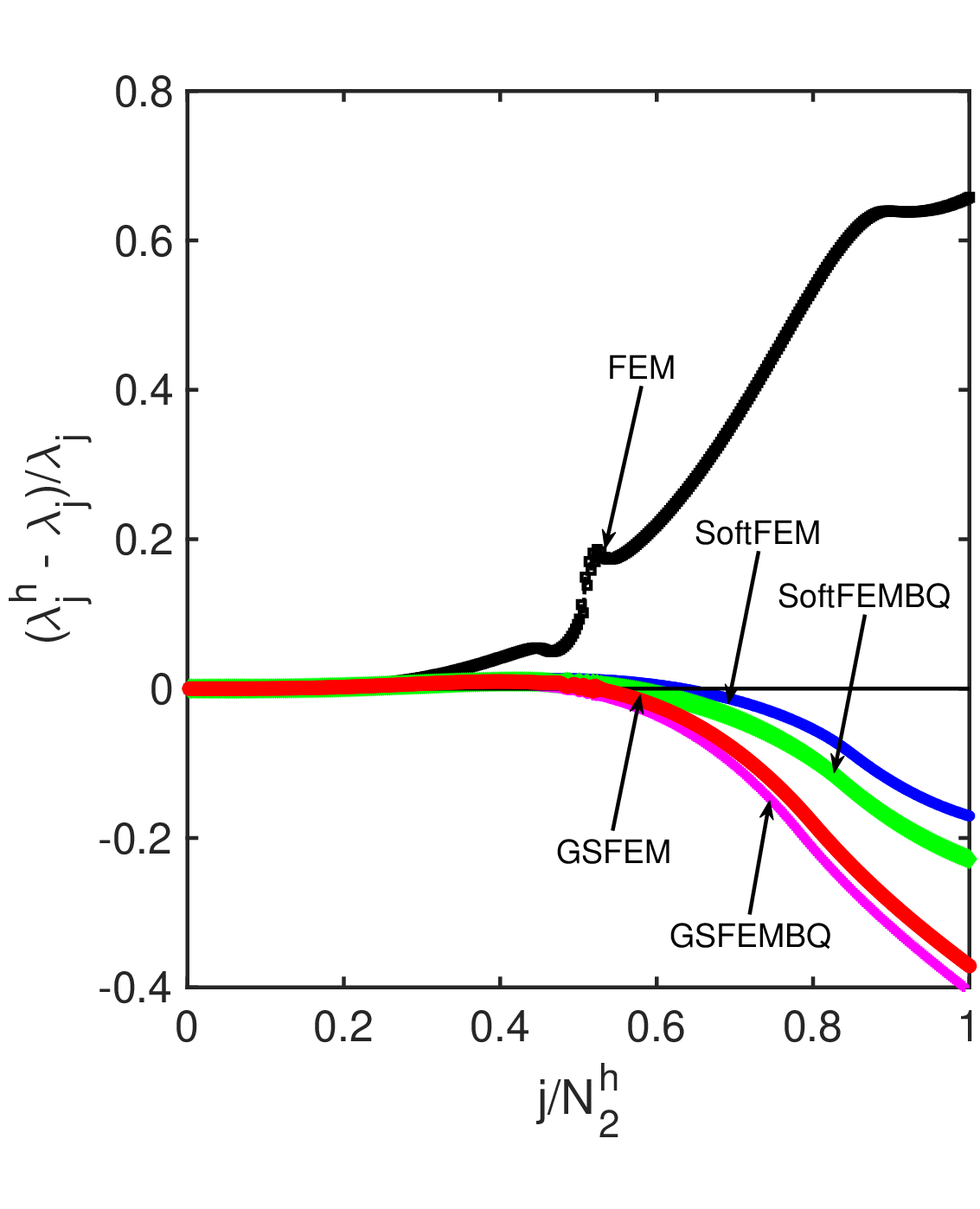}
    \includegraphics[width=0.32\textwidth]{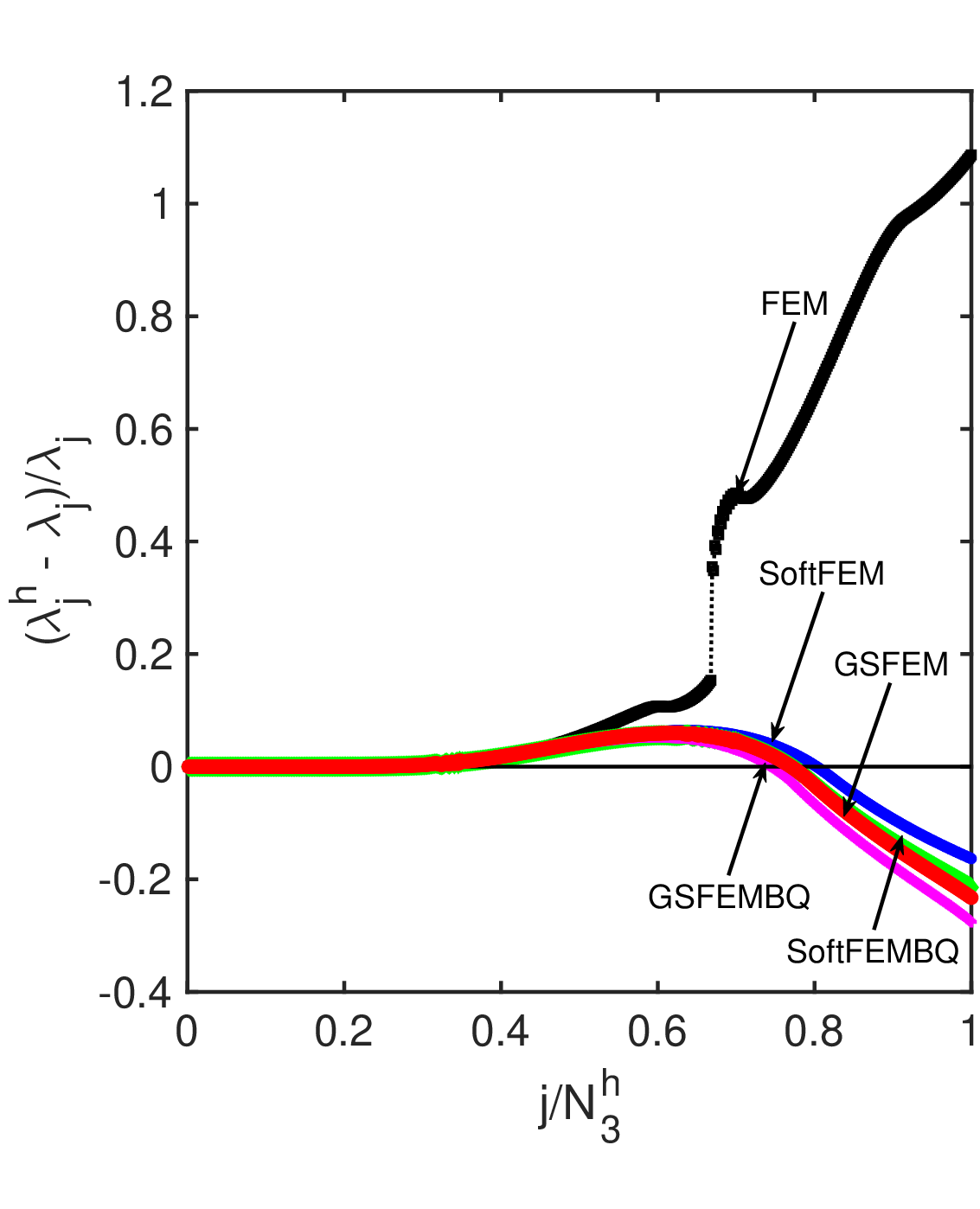}

    \caption{
    Comparison of eigenvalue errors using different methods for~\eqref{eq:pde} with $\kappa=e^{x+x^2}$ and $N^h=200$ uniform elements: Galerkin FEM (black), SoftFEM (blue), SoftFEMBQ (green), and GSFEMBQ (magenta). The parameters are chosen as in Table~\ref{tab:result1} and the blending parameter $\alpha=0.95$. }
    \label{fig:dc_1}
\end{figure}

\begin{table}[h!]
    \centering
    \footnotesize
    \begin{tabular}{C{0cm}C{0.8cm}C{0.8cm}C{0.8cm}C{0.8cm}C{0.8cm}C{0.8cm}C{0.8cm}C{0.8cm}C{0.8cm}C{0.8cm}C{0.8cm}C{0.8cm}}
    \toprule
    $p$ & $\lambda_{min}^h$ & $\lambda_{max}^h$ & $\lambda_{s,max}^h$ & $\lambda_{gs,max}^h$ & $\lambda_{sq,max}^h$ & $\lambda_{gsq,max}^h$ & $\sigma$ & $\sigma_{s}$ & $\sigma_{gs}$ & $\sigma_{sq}$ & $\sigma_{gsq}$ & $\rho_{gsq}$ \\
    \midrule

    $1$ & $11.05$ & $6.14e5$ & $4.09e5$ & $3.50e5$ &$3.72e5$ & $3.22e5$ & $5.55e4$ & $3.70e4$ & $3.16e4$ & $3.37e4$ & $2.92e4$ & $1.90$\\[1mm]

    $2$ & $11.05$ & $3.07e6$ & $1.54e6$ &$1.17e6$ &$1.43e6$ & $1.10e6$ & $2.78e5$ & $1.39e5$ & $1.05e5$ & $1.29e5$ & $9.98e4$ & $2.79$\\[1mm]
    
    $3$ & $11.05$ & $8.72e6$ & $3.50e6$ &$3.21e6$ &$3.28e6$ & $3.03e6$ & $7.89e5$ &$3.16e5$ & $2.90e5$ & $2.97e5$ & $2.74e5$ &$2.88$\\
    
    \bottomrule
    \end{tabular}
    \caption{Comparison of key metrics - maximum eigenvalue, condition number, and stiffness reduction ratios - among Galerkin FEM, SoftFEM, GSFEM, SoftFEMBQ, and GSFEMBQ methods with $p=1, 2, 3$ and a mesh of $N^h=200$ for~\eqref{eq:pde} with $\kappa=e^{x+x^2}$ in 1D. The parameters are in Table~\ref{tab:result1} and a blending parameter $\alpha=0.95$.}
    \label{tab:dc_1_1}
\end{table}




\section{Concluding Remarks} 
\label{sec:conclusion} 

This study generalizes the SoftFEM with two main ideas focusing on the mass bilinear form, aiming at further reducing stiffness and improving accuracy, particularly in the higher spectral range. 
The generalization is twofold: adding an auxiliary bilinear form for gradient jumps and using blended quadratures. 
Both methods improve the spectral accuracy and reduce the condition numbers. 
There are several lines for future work. 
Firstly, the analytical description of the matrix eigenvalue problems arising from higher-order elements generalizes the linear elements studied in this paper.
This sharp description will contribute to the optimal choice of the parameters.
Secondly, nonlinear eigenvalue problems arise in various engineering applications; therefore, applying the proposed method to nonlinear eigenvalue problems may prove beneficial. 
Thirdly, the idea can be generalized to soft isogeometric analysis \cite{deng2023softiga}.
Last, the methods can be applied to explicit time discretization of non-stationary differential problems to increase the stability region of the explicit time marching schemes.




\bibliographystyle{plain}
\bibliography{ref}





\end{document}